\documentclass[letterpaper, 10 pt, conference]{ieeeconf}  
\IEEEoverridecommandlockouts                              

\usepackage[]{graphicx}
\usepackage{amsmath}
\usepackage{amssymb}
\usepackage{bm}
\usepackage{cases}
\usepackage{subcaption}
\usepackage{cite}
\usepackage{ascmac}
\usepackage{comment}
\usepackage{setspace}
\usepackage{algorithm}

\newcommand{\tr}{\mathsf{T}}

\newtheorem{prop}{Proposition}
\newtheorem{prob}{Step}
\newtheorem{remark}{Remark}
\newtheorem{col}{Corollary}

\usepackage[hidelinks]{hyperref}

\title{\LARGE \bf
	Control-aware Learning of Koopman Embedding Models \\
}

\author{Daisuke Uchida$^{1}$ and Karthik Duraisamy$^{1}$
	\thanks{$^{1}$The authors are with the department of Aerospace Engineering, University of Michigan, Ann Arbor, Michigan, 48109, USA
		(e-mail: duchida@umich.edu, kdur@umich.edu).
	}
}

\begin{document}

	\maketitle
	\thispagestyle{empty}
	\pagestyle{empty}

	\begin{abstract}
		A learning method is proposed for Koopman operator-based models with the goal of improving closed-loop control behavior. A neural network-based approach is used to discover a space of observables in which nonlinear dynamics is linearly embedded. While accurate state predictions can be expected with the use of such complex state-to-observable maps, undesirable side-effects may be introduced when the model is deployed in a closed-loop environment. This is because of modeling or residual error in the linear embedding process, which can manifest itself in a different manner compared to the state prediction. 
		To this end, a technique is proposed to refine the originally trained model with the goal of improving the closed-loop behavior of the model while retaining the  state-prediction accuracy obtained in the initial learning.  Finally, a simple data sampling strategy is proposed to use inputs deterministically sampled from continuous functions, leading to additional improvements in the controller performance for nonlinear dynamical systems. Several numerical examples are provided to show the efficacy of the proposed method.
	\end{abstract}
	\vspace{-0mm}
	
	\section{INTRODUCTION}
	\vspace{-0mm}
	The Koopman operator describes the evolution of a nonlinear dynamical system in terms of a possibly infinite-dimensional but linear operator in a lifted space of observables.
	Starting from applications in dimension reduction of high-dimensional nonlinear systems such as turbulent flows\cite{Spectral_analysis_of_nonlinear_flows,P_shmidt_DMD_2008}, the Koopman operator has gained high popularity in recent years as a data-driven modeling approach for dynamical systems.
	Among a number of applications is that of data-driven control.
	On the basis of finite-dimensional approximation of the Koopman operator derived from Extended Dynamic Mode Decomposition (EDMD)\cite{Williams2015}, which yields a Linear Time-Invariant (LTI) data-driven model, several linear controllers have been applied such as Linear Quadratic Regulator (LQR)\cite{derivative_based_Murphy,local_Koopman,soft_robot_arm} and Model Predictive Control (MPC)\cite{KORDA_Koopman_MPC,Koopman_Lyapunov_based_MPC,Koopman_generators_Peitz,Soft_robot,tube_based_MPC}.
	Also, to overcome a limitation of the Koopman formalism that the modeling error is practically inevitable due to the nature of purely data-driven modeling procedures, several concepts from control theories, e.g., robust tube-based MPC\cite{tube_based_MPC}, offset-free MPC\cite{handling_plant_model_mismatch}, conformant synthesis\cite{conformant_systhesis}, and so on, have been utilized to achieve the control objectives in the fully data-driven setting. 
	While there are many frameworks and methods in literature to incorporate the Koopman operator into data-driven control, not much attention has been paid to the modeling aspect of the Koopman operator-based framework itself, i.e.,
	how the modeling error behaves while implementing state-prediction or feedback control, 
	which pertains to practically unavoidable discrepancy between theories and actual implementations and should be considered as an essential factor for reliable modeling procedures and control applications.

	It has been recognized that the convergence property of the EDMD algorithm\cite{on_convergence_of_EDMD} does not hold for non-autonomous types of Koopman models\cite{KORDA_Koopman_MPC}. 
	This
	motivates the use of neural networks to learn the observable functions themselves along with the finite-dimensional approximation of the Koopman operator.
		Especially, while the model structure we employ in this paper is restrictive so that we can realize LTI systems in the embedded space, which allows the use of linear controller designs even if the dynamics is nonlinear,
		it is shown that obtaining high state-predictive accuracy is achievable if the dynamics is linear with respect to input and one has access to enough data and computational resources that afford high-dimensional and complex feature maps. 

	On the other hand, we show that the modeling error of the Koopman models interacts with the closed-loop system in a different way from the state-prediction and we exemplify that the controller performance can greatly suffer from the modeling error, on which the complexity and dimension of observables have a large influence.
	To improve the possibly undesirable closed-loop behavior induced by Koopman-based control models, a control-aware method is proposed, in which the model is refined after the initial training with additional use of data points sampled from closed-loop dynamics. 
	This modification of the model aims to directly reduce the impact of the modeling error on the controller performance.
	Moreover, with the same intent as the control-aware learning method, we also present a simple yet effective data sampling strategy that only uses inputs deterministically sampled from continuous functions.

	This paper is organized as follows. In Section \ref{section 2}, the Koopman operator framework for non-autonomous systems is presented. 
	In Section \ref{section 3}, we discuss the manifestation of the modeling error on both prediction and control and propose a control-aware learning method along with a data sampling strategy about inputs to improve the actual closed-loop behavior.
	Finally, several dynamical systems are tested to show the effectiveness of the proposed method in Section \ref{section. numerical examples}.

	\vspace{-0mm}
	\section{KOOPMAN OPERATOR THEORY FOR NON-AUTONOMOUS SYSTEMS}
	\vspace{-0mm}
	\label{section 2}
		\subsection{Koopman Operator for General Non-Autonomous Systems}
		In this section, the Koopman operator is formally introduced for general non-autonomous systems, which forms the basis of observations made in Section \ref{section 3} and the proposed method.
	Consider a dynamical system:
	\vspace{-0mm}
	\begin{align}
		\dot{x}(t)=f(x(t),u(t)),
		\label{eq. governing eq in ODE}
	\end{align}
	where $x(t)\in \mathcal{X}\subseteq \mathbb{R}^n$, $u(t)\in \mathcal{U}\subseteq \mathbb{R}^p$, and $f:\mathcal{X}\times \mathcal{U}\rightarrow \mathbb{R}^n$ are the state, the input, and the possibly nonlinear mapping describing dynamics of the system, respectively.
	Throughout the paper, we assume the solution $x(t)$ to \eqref{eq. governing eq in ODE} to be continuous with respect to $t$.
	With a first-order time discretization, \eqref{eq. governing eq in ODE} yields the following difference equation:
	\vspace{-0mm}
	\begin{align}
		x_{k+1}=F(x_k, u_k),
		\label{eq. governing eq}
	\end{align}
	where $x_k:=x(k\Delta t)$, $u_k:=u(k\Delta t)$, and $\Delta t$ denotes the sampling period.
	On the assumption $\Delta t\ll 1$, we consider \eqref{eq. governing eq} as the discrete-time system whose dynamics is equivalent to that of \eqref{eq. governing eq in ODE}.
	It is assumed that $f$ (i.e., $F$) is unknown and its dynamics is modeled in a data-driven manner.

	In the Koopman operator formalism, the dynamics is characterized through functions called observables, which are mappings from the state-space into $\mathbb{R}$.
	While the Koopman operator was first introduced in the context of autonomous systems, there have been also several efforts extending it to non-autonomous systems with control 
	inputs\cite{DMDc,EDMDc,KORDA_Koopman_MPC}. 
	In a formal extension\cite{KORDA_Koopman_MPC},
	the state-space is extended to the augmented space $\mathcal{X}\times l (\mathcal{U})$, 
	where 
	\vspace{-0mm}
		\begin{align}
			\label{eq. def of space of sequenses of inputs}
			l (\mathcal{U}):=\{ \bm{U}:=(u_1,u_{2},\cdots) \mid u_i\in \mathcal{U},\forall i \in \mathbb{N} \},
		\end{align}
	is the space of sequences of inputs,
	and the observables $g$ are of the form:
	\vspace{-0mm}
	\begin{align}
		g:\mathcal{X}\times l (\mathcal{U})
		\rightarrow 
		\mathbb{R}:
		(x_k,\bm{U})
		\mapsto
		g(x_k,\bm{U}).
		\label{eq. obs general description}
	\end{align}
	In practice, the observables $g$ may be considered as the feature maps that are either specified by users or learned from data in the modeling procedure.

	The Koopman operator corresponding to the non-autonomous system \eqref{eq. governing eq} is defined as an infinite-dimensional linear operator $\mathcal{K}:\mathcal{F}\rightarrow \mathcal{F}$ ($\mathcal{F}$: space of functions $g$) s.t.
	\vspace{-0mm}
	\begin{align}
		&\mathcal{K}g = g\circ \hat{F}
		\ \Leftrightarrow\ 
		(\mathcal{K}g)(x_k, \bm{U})=
		g(\hat{F}(x_k, \bm{U})),
		&
		\label{eq. def of Koopman operator}
	\end{align}
	where 
	the mapping $\hat{F}:\mathcal{X}\times l (\mathcal{U})\rightarrow \mathcal{X}\times l (\mathcal{U})$ is defined by
	\vspace{-0mm}
	\begin{align}
		\label{eq. def of augmented dynamics}
		\hat{F}(x_k, \bm{U})
		:=
		\left(F(x_k,u_k), \mathcal{S}\bm{U}\right)
		=
		\left(
		x_{k+1}, \mathcal{S}\bm{U}
		\right),
	\end{align}
		with the notation
		$\bm{U}=(u_k,u_{k+1},\cdots)$,
	and $\mathcal{S}$ denotes the shift operator s.t.
	\vspace{-0mm}
	\begin{align}
		\label{eq. def of shift operator}
		\mathcal{S}\bm{U}=
		\mathcal{S}(u_k,u_{k+1},\cdots):=
		(u_{k+1},u_{k+2},\cdots).
	\end{align}

	It is easily inferred from \eqref{eq. def of Koopman operator} that $\mathcal{K}$ is a linear operator.
Note that since \eqref{eq. governing eq} only specifies the evolution of $x_k$, it is required to introduce the sequence $\bm{U}=(u_k,u_{k+1},\cdots)$ of inputs, which can be also interpreted as an input signal (i.e., a function) $U:\mathbb{Z}_{\geq 0}\rightarrow \mathcal{U}$,
	to formally define the Koopman operator $\mathcal{K}$.
	The equation \eqref{eq. def of Koopman operator} along with the definition \eqref{eq. def of augmented dynamics} can be viewed as the evolution of the dynamics \eqref{eq. governing eq} through the observables $g$.
	\vspace{-0mm}
	\subsection{Finite Dimensional Approximation of the Koopman Operator and Data-Driven Koopman Models}
	\vspace{-0mm}
	To apply the Koopman operator formalism to dynamical systems modeling, a finite-dimensional approximation $K$ of the Koopman operator $\mathcal{K}$ is introduced as follows.
	\vspace{3mm}
	\begin{prop}
		\label{prop. invariant subspace and existence of finite dimensional K}
		\rm{}
		Given observables $g_i\in \mathcal{F}$ ($i=1,\cdots,D$),
		let $g$ be an arbitrary element of $\text{span}(g_1\cdots, g_D)$. 
		Then, $\mathcal{K}g\in \text{span}(g_1\cdots, g_D)$, i.e., $\text{span}(g_1\cdots, g_D)$ is an invariant subspace under the action of the Koopman operator $\mathcal{K}:\mathcal{F}\rightarrow \mathcal{F}$, if and only if there exits $K\in \mathbb{R}^{D\times D}$ s.t.
		\begin{align}
			[\mathcal{K}g_1\ \cdots \ \mathcal{K}g_D]^\tr
			= K[g_1\ \cdots \ g_D]^\tr.
			\label{eq. finite dimensional approx of K}
		\end{align}
	\end{prop}
	\begin{proof}
		See Appendix \ref{appendix proof of prop 1}.
	\end{proof}
	\vspace{3mm}

	From an engineering perspective, it is of great interest to introduce the observables such that they allow practical models for control applications. One major choice of $g_i$ for the Koopman control problem takes the following structure of observables\cite{Model_based_control,Soft_robot,tube_based_MPC}:
	\vspace{-0mm}
	\begin{align}
		[g_1(x_k,\bm{U})\cdots g_{D}(x_k,\bm{U})]^\tr
		=
		\left[
		x_k^\tr\ \tilde{g}(x_k)^\tr\ u_k^\tr
		\right]^\tr,
		\label{eq. def of obs}
	\end{align}
	where $D=n+N+p$ and $\tilde{g}(x_k)\in \mathbb{R}^N$ represents a vector-valued function from $\mathcal{X}$ into $\mathbb{R}^N$ for some $N\in \mathbb{N}$.
	Note that only the first element $u_k$ in the sequence $\bm{U}=(u_k,u_{k+1},\cdots)$ appears in the definition \eqref{eq. def of obs}, which leads to a practical form of data-driven models consistent with many linear controller designs such as LQR and MPC.
	On the assumption that we have access to $x_k$ and $u_k$ as data, we consider the following finite-dimensional approximation $K_c\in\mathbb{R}^{(n+N+p)\times (n+N+p)}$ of the Koopman operator $\mathcal{K}$:
	\vspace{-0mm}
	\begin{align}
		\left[
		\begin{array}{c}
			x_{k+1}
			\\	
			\tilde{g}(x_{k+1})
			\\
			u_{k+1}
		\end{array}
		\right]
		\approx
		\underset{=:K_c}{
			\underbrace{
				\left[
				\begin{array}{c}
					\begin{array}{cc}
						A & B
					\end{array}
					\\
					*
				\end{array}
				\right]
		}}
		\left[
		\begin{array}{c}
			x_{k}
			\\	
			\tilde{g}(x_{k})
			\\
			u_{k}
		\end{array}
		\right],
		\label{eq. intro to Koopman model}
	\end{align}
	where matrices $A\in\mathbb{R}^{(n+N)\times (n+N)}$ and $B\in\mathbb{R}^{(n+N)\times p}$ are to be learned along with the feature maps $\tilde{g}$.
	Note that \eqref{eq. intro to Koopman model} is approximate since $\text{span}(g_1,\cdots,g_D)$ defined by \eqref{eq. def of obs} may not be invariant under the action of $\mathcal{K}$.

	Noticing that the first $n+N$ rows of \eqref{eq. intro to Koopman model} are enough to specify the evolution of the state $x_k$ s.t.
	\vspace{-0mm}
	\begin{align}
		\left[
		\begin{array}{c}
			x_{k+1}
			\\	
			\tilde{g}(x_{k+1})
		\end{array}
		\right]
		\approx
		A
		\left[
		\begin{array}{c}
			x_{k}
			\\	
			\tilde{g}(x_{k})
		\end{array}
		\right]
		+
		Bu_k,
		\label{eq. Koopman model 1}
	\end{align}
	we are only interested in learning \eqref{eq. Koopman model 1} and the last $p$ rows of $K_c$ in \eqref{eq. intro to Koopman model} are ignored in the proceeding formulations. 
		The discarded equations approximately represent $u_{k+1}$ given $x_k$ and $u_k$, which corresponds to the fact that the Koopman operator shifts the sequence of inputs according to \eqref{eq. def of shift operator}.
	From \eqref{eq. Koopman model 1}, the modeling error is defined as:
	\vspace{-0mm}
	\begin{align}
		r(x,u):=
		\left[
		\begin{array}{c}
			F(x,u)
			\\	
			\tilde{g}(F(x,u))
		\end{array}
		\right]
		-
		\left(
		A
		\left[
		\begin{array}{c}
			x
			\\	
			\tilde{g}(x)
		\end{array}
		\right]
		+
		Bu
		\right),
		\label{eq. def of modeling error}
	\end{align}
	and its norm:
	\vspace{-0mm}
	\begin{align}
		\| r \|_{L_2}=
		\sqrt{
			\int_{\mathcal{X}\times \mathcal{U}} 
			\| r(x,u) \|_2^2 dxdu 
		},
		\label{eq. norm of r}
	\end{align}
	may be used as a characteristic to evaluate the model, e.g., \eqref{eq. Koopman model 1} is exact almost everywhere if \eqref{eq. norm of r} is well-defined and $\|r\|_{L_2}=0$.

	The model \eqref{eq. Koopman model 1} is an LTI system in the new coordinates $[x_k^\tr\ \tilde{g}(x_k)^\tr]^\tr$ and linear controller designs can be applied to control \eqref{eq. governing eq}.
	In this paper, the following feedback controller with a static gain $\bm{K}\in \mathbb{R}^{p\times (n+N)}$ is considered:
	\vspace{-0mm}
	\begin{align}
		u_k=\bm{K}
		[x_k^\tr\ \tilde{g}(x_k)^\tr]^\tr.
		\label{eq. controller}
	\end{align}

	\section{Control-Consistent Learning of Koopman Embedding}
	\vspace{-0mm}
	\label{section 3}
	\subsection{Motivating Example}
	\label{section. motivating example}
	We consider the following one-dimensional system as a guiding example to motivate the proposed control-aware learning.
	\vspace{-0mm}
	\begin{align}
		x_{k+1}={x^2_k} e^{-x_k}+u_k,
		\ \ 
		x_k,u_k\in \mathbb{R}.
	\end{align}

	Suppose that we create the Model 1 defined as:
	\vspace{-0mm}
	\begin{align}
		&
		\left[
		\begin{array}{c}
			x_{k+1}
			\\
			\hspace{-2mm}{x^2_{k+1}}e^{-x_{k\hspace{-0.5mm}+\hspace{-0.5mm}1}}\hspace{-2.5mm}
		\end{array}
		\right]
		\hspace{-1.5mm}\approx \hspace{-1mm}
		A\hspace{-1mm}
		\left[
		\begin{array}{c}
			x_{k}
			\\
			\hspace{-2mm}{x^2_k}e^{-x_k}\hspace{-2.5mm}
		\end{array}
		\right]
		\hspace{-1.5mm}+\hspace{-1mm}
		Bu_k,\hspace{-0.5mm}
		\left(
		\tilde{g}(x_k)\hspace{-1mm}=\hspace{-1mm}
		{x^2_k}e^{-x_{k}}
		\right)\hspace{-1mm}.
		&
		\label{eq. ex1 model 1}
	\end{align}

	From Proposition \ref{prop. state prediction} in Section \ref{section. second training}, perfect state prediction with no error is possible with $A$ and $B$ given as the following forms:
	\vspace{-0mm}
	\begin{align}
		A=
		\left[
		\begin{array}{cc}
			0 & 1
			\\
			\alpha_1 & \alpha_2
		\end{array}
		\right],
		B=
		\left[
		\begin{array}{c}
			1
			\\
			\alpha_3
		\end{array}
		\right],
		\alpha_i\in \mathbb{R}.
	\end{align}

	The modeling error \eqref{eq. def of modeling error} is then represented by
	\vspace{-0mm}
	\begin{align}
		r(x_k,u_k)\hspace{-1mm}=\hspace{-1mm}
		\left[
		\begin{array}{c}
			0
			\\
			\begin{array}{l}
				\hspace{-4mm}
				({x^2_k}e^{-x_k} + u_k)^2
				\exp(-{x^2_k}e^{-x_{k}} \hspace{-1mm}-\hspace{-1mm} u_k)
				\\
				\hspace{14mm}-\hspace{-0.1mm} \alpha_1 x_k \hspace{-1mm}-\hspace{-1mm} \alpha_2 {x^2_k}e^{-x_{k}} \hspace{-1mm}-\hspace{-1mm} \alpha_3 u_k
				\hspace{-4mm}
			\end{array}
		\end{array}
		\right]\hspace{-1mm}.
		\label{eq. r1}
	\end{align}

	Suppose we also have the Model 2, which has richer features:
	\vspace{-0mm}
	\begin{align}
		&
		\left[
		\begin{array}{c}
			x_{k+1}
			\\
			\hspace{-2.5mm}{x^2_{k+1}}e^{-x_{k\hspace{-0.5mm}+\hspace{-0.5mm}1}}\hspace{-2.8mm}
			\\
			x_{k+1}^2
		\end{array}
		\right]
		\hspace{-1.5mm}\approx\hspace{-1mm} 
		A\hspace{-1.5mm}
		\left[
		\begin{array}{c}
			x_{k}
			\\
			\hspace{-2.5mm}{x^2_k}e^{-x_{k}}\hspace{-2.8mm}
			\\
			x_{k}^2
		\end{array}
		\right]
		\hspace{-2mm}+\hspace{-1mm}
		Bu_k,	\hspace{-1mm}
		\left(
		\hspace{-1mm}
		\tilde{g}(x_k)
		\hspace{-1mm}=\hspace{-1.5mm}
		\left[
		\begin{array}{c}
			\hspace{-2.3mm}	{x^2_k}e^{\hspace{-0.5mm}-\hspace{-0.2mm}x_{k}}\hspace{-2.5mm}
			\\
			x_k^2
		\end{array}
		\right]
		\hspace{-0.5mm}
		\right)\hspace{-1.2mm}.
		&
		\label{eq. ex1 model 2}
	\end{align}
	Perfect state prediction can be also achieved with
	\vspace{-0mm}
	\begin{align}
		A=
		\left[
		\begin{array}{ccc}
			0 & 1 & 0
			\\
			\beta_1 & \beta_2 & \beta_3
			\\
			\beta_4 & \beta_5 & \beta_6
		\end{array}
		\right],
		B=
		\left[
		\begin{array}{c}
			1
			\\
			\beta_7
			\\
			\beta_8 
		\end{array}
		\right],
		\beta_i\in \mathbb{R}.
	\end{align}

	In this case, however, the modeling error $r(x_k,u_k)$ takes a different form:
	\vspace{-0mm}
	\begin{align}
		r(x_k,\hspace{-0.7mm}u_k)\hspace{-1mm}=\hspace{-1.5mm}
		\left[
		\begin{array}{c}
			0
			\\
			\begin{array}{l}
				\hspace{-4mm}
				({x^2_k}e^{-x_k} + u_k)^2
				\exp(-{x^2_k}e^{-x_{k}}-u_k)  
				\\
				\hspace{7mm} 
				- \beta_1 x_k - \beta_2 {x^2_k}e^{-x_{k}} - \beta_3 x_k^2 - \beta_7 u_k\hspace{-4mm}
			\end{array}
			\\
			\begin{array}{l}
				\hspace{-4mm}({x^2_k}e^{-x_{k}}+u_k)^3
				- \beta_5 x_k 
				\\
				\hspace{18mm} - \beta_6 {x^2_k}e^{-x_{k}} - \beta_7 x_k^2 - \beta_8 u_k\hspace{-4mm}
			\end{array}
		\end{array}
		\right]\hspace{-1mm}.
		\label{eq. r2}
	\end{align}

	Figures \ref{subfig. ex1 heat map r1} and \ref{subfig. ex1 heat map r2} show the heat maps of $\|r(x_k,u_k)\|_2$, where the model parameters $\alpha_i$ and $\beta_i$ are obtained by EDMD\cite{KORDA_Koopman_MPC}.
	Model 2 in \eqref{eq. ex1 model 2} is highly erroneous for $x_k<0$ compared to the Model 1 in \eqref{eq. ex1 model 1}
	and the modeling error 
	accumulates according to \eqref{eq. error accumulation of control model} in Section \ref{section. second training} leading to undesirable closed-loop behaviors.
	Figure \ref{fig. control performance ex1} shows the controller performance of both models, where $\bm{K}$ in \eqref{eq. controller} is computed as an LQR gain with the cost function $\sum_{k=0}^{\infty}{x^2_k}+{u^2_k}$.
	Despite the fact that both models achieve precisely zero state-prediction error, the controller designed for Model 2 causes an undesirable oscillation even as the open-loop dynamics quickly converges to the origin after $k=1$.
	Herein, it is also emphasized that from the state-prediction point of view, complex and large models such as the Model 2 may be preferable in general since it is more likely to achieve better state-prediction by Corollary \ref{col. state prediction} in Section \ref{section. second training}.
	To deal with this degradation of the controller performance, which is not revealed from the state-prediction accuracy, a control-aware learning approach is proposed along with a simple data sampling strategy.
	\vspace{-0mm}
	\begin{figure}
		\centering
		\begin{subfigure}{0.45\linewidth}
			\centering
			\includegraphics[width=0.95\linewidth]{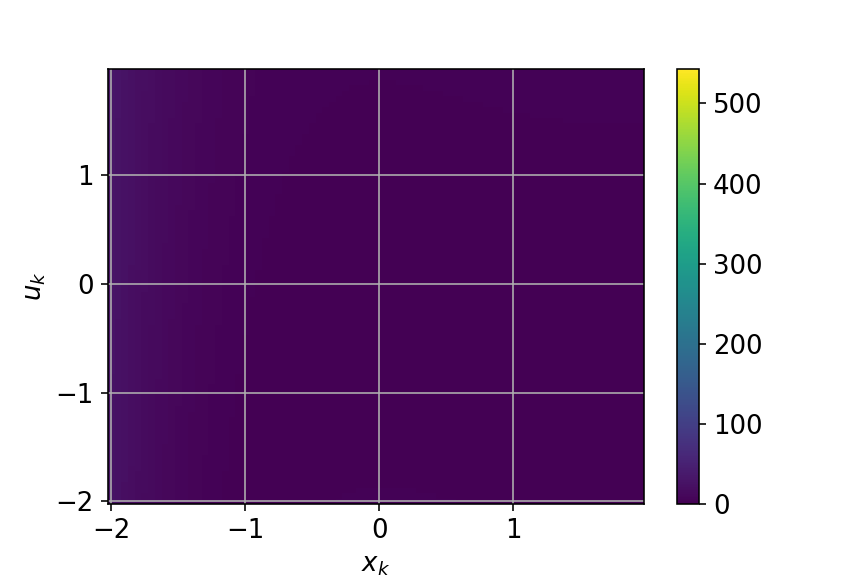}
			\caption{Model 1 \eqref{eq. r1}.}
			\label{subfig. ex1 heat map r1}
		\end{subfigure}
		\begin{subfigure}{0.45\linewidth}
			\centering
			\includegraphics[width=0.95\linewidth]{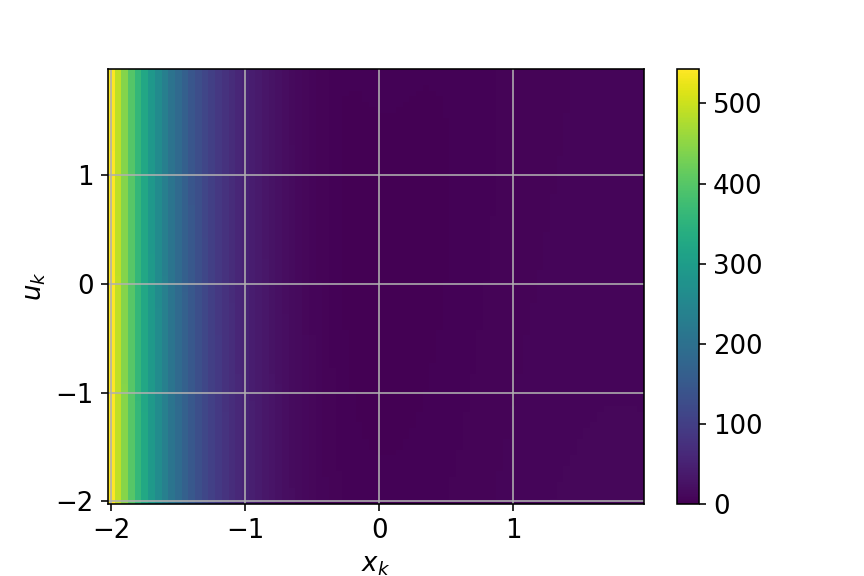}
			\caption{Model 2 \eqref{eq. r2}.}
			\label{subfig. ex1 heat map r2}
		\end{subfigure}
		\begin{subfigure}{0.45\linewidth}
			\centering
			\includegraphics[width=0.95\linewidth]{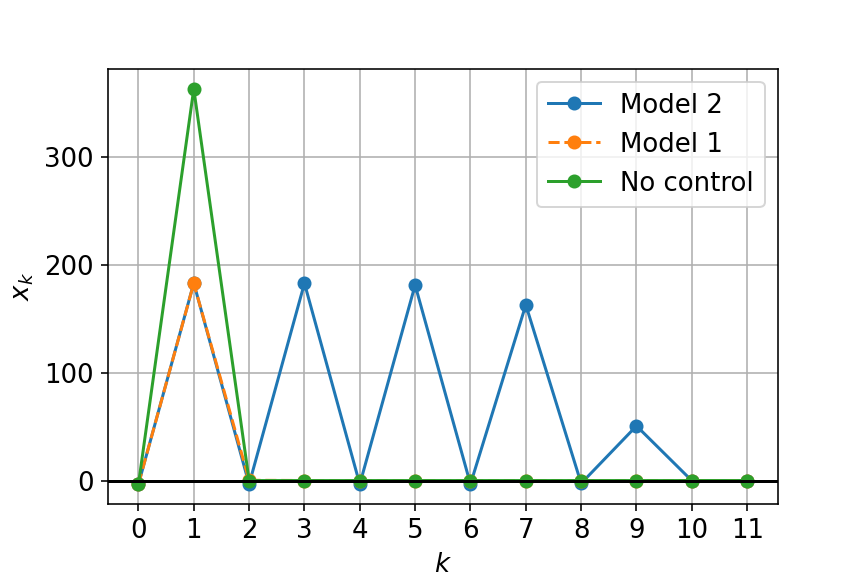}
			\caption{Control simulations.}
			\label{fig. control performance ex1}
		\end{subfigure}
		\caption{\em (a), (b): Heat maps of $\|r(x_k,u_k)\|_2$. (c): Controller performance with $x_0=-3.4298$.}
		\vspace{-0mm}
	\end{figure}
	
	\vspace{-0mm}
	\subsection{Modeling Errors of Koopman Control Models}
	\vspace{-0mm}
	Several methods have been proposed to learn the model \eqref{eq. Koopman model 1}, among of which the most straightforward one is to first specify the feature maps $\tilde{g}$ and infer $A$ and $B$ from data.
	This is reduced to a linear regression problem and a unique solution can be obtained in the same manner as the Dynamic Mode Decomposition (DMD)\cite{P_shmidt_DMD_2008,Spectral_analysis_of_nonlinear_flows}.
	This learning algorithm with nonlinear observables is called the Extended Dynamic Mode Decomposition (EDMD)\cite{Williams2015}, based on which many data-driven Koopman controller designs are developed\cite{KORDA_Koopman_MPC,Koopman_Lyapunov_based_MPC,Model_based_control,PEITZ_switched_control,Soft_robot,data-driven_Koopman_H2,local_Koopman,soft_robot_arm,tube_based_MPC,Koopman_generators_Peitz}.
	An important feature of the EDMD algorithm is its convergence property.
	Letting $\mathcal{F}=L_2(\mathcal{X}\times l(\mathcal{U}))$,
	the approximation obtained by EDMD is, under several assumptions, shown to converge to the true Koopman operator in the strong operator topology as the number $M$ of data points and the number $D$ of observables $g_i$ tend to infinity\cite{On_convergence_Klus,on_convergence_of_EDMD}.
	Specifically, this convergence property is stated as follows:
	for $\forall g\in L_2(\mathcal{X}\times l(\mathcal{U}))$,
	\vspace{-0mm}
	\begin{align}
		\lim_{D\rightarrow \infty}
		&\int_{\mathcal{X}\times l(\mathcal{U})} |(P_D^\mu \mathcal{K})P_D^\mu g - \mathcal{K}g|^2d\mu 
		=0,
		&
		\label{eq. convergence of EDMD 1}
	\end{align}
	where $P_D^\mu $ and $\mu$ are the $L_2$ projection onto $\text{span}(g_1,\hspace{-0.5mm}\cdots\hspace{-0.5mm},g_D)$ and a measure with which $L_2(\mathcal{X}\times l(\mathcal{U}))$ is endowed, respectively. 
	The operator $P_D^\mu \mathcal{K}$ in \eqref{eq. convergence of EDMD 1} is related to the finite dimensional approximation $K$ computed by EDMD in the following manner:
	\vspace{-0mm}
	\begin{align}
		&\lim_{M\rightarrow \infty} \|  a^\tr K[g_1\cdots g_D]^\tr - P_D^{\mu}\mathcal{K}g\|=0,&
		\nonumber
		\\
		&
		\forall g=a^\tr [g_1\cdots g_D]^\tr \in \text{span}(g_1,\cdots,g_D),
		a\in \mathbb{R}^D,
		&
		\label{eq. convergence of EDMD 2}
	\end{align}
	where $\|\cdot \|$ is an arbitrary norm on $\text{span}(g_1,\cdots,g_D)$.
	The convergence property \eqref{eq. convergence of EDMD 1} implies that data-driven Koopman models can provide reasonable approximation with sufficiently large $M$ and $D$.
	However, it is not the case for non-autonomous systems as seen in the following remark.  
	\vspace{3mm}
	\begin{remark}
		\label{remark convergenve property does not hold}
		Given feature maps $\tilde{g}(x_k)$,
		the approximation $K_c$ in \eqref{eq. intro to Koopman model} does not possess
		the the convergence property \eqref{eq. convergence of EDMD 1} since for \eqref{eq. convergence of EDMD 1} to hold, it is necessary that $\{g_i\}_{i=1}^D$ is an orthonormal basis of $L_2(\mathcal{X}\times l(\mathcal{U}))$ as $D\rightarrow \infty$\cite{on_convergence_of_EDMD}. 
		As described in \cite{KORDA_Koopman_MPC}, no elements of $\bm{U}=(u_k,u_{k+1},\cdots)$ except for the first one $u_k$ depend on the definition \eqref{eq. def of obs} and it is obvious that they cannot form any basis of $L_2(\mathcal{X}\times l(\mathcal{U}))$.	
	\end{remark}
	\vspace{3mm}

	No convergence property of the model \eqref{eq. intro to Koopman model} or \eqref{eq. Koopman model 1} implies $\|r\|_{L_2}$ may not be negligibly small even if $D$ and $M$ are sufficiently large.
	To this end, we adopt another learning formulation, where the nonlinear feature maps $\tilde{g}$ are also learned from data along with matrices $A$ and $B$.
	\vspace{-0mm}
	
	\subsection{Initial Training: Simultaneous Learning of the Feature Maps and the System Matrices}
	\vspace{-0mm}
	Another class of methods to learn model \eqref{eq. Koopman model 1} estimates the system matrices $A,B$, and the feature maps $\tilde{g}$ simultaneously.
	The resulting models are expected to achieve better predictive accuracy than those of the linear formulations such as EDMD since they can have greater model expressivity with the feature maps $\tilde{g}$ also learned from data along with the matrices $A$ and $B$.
	Especially, the use of neural networks has been shown to be promising to incorporate into the Koopman operator-based modeling, analyses, and control\cite{Learning_DNN_ACC2019,Learning_Koopman_Invariant_Subspaces,Physics-based_robabilistic_learning,deep_learning_representation_CDC,Wiener_MIMO_ID}.

	Hence, the proposed method characterizes the feature maps $\tilde{g}$ in \eqref{eq. Koopman model 1} as a neural network aiming at high predictive accuracy and solves a nonlinear regression problem to learn $A,B$, and $\tilde{g}$, which is formulated as Step \ref{problem. initial training}.
	\begin{table}
		\begin{screen}
			\begin{prob}
				\label{problem. initial training}
				\rm{} \ (Initial Training) 
				\\ 
				Find $g$, $A\in\mathbb{R}^{(n+N)\times (n+N)}$, and $B\in \mathbb{R}^{(n+N)\times p}$ s.t.
				\vspace{-0mm}
				\begin{align}
					&\left\{
					g,\ A,\ B
					\right\}=
					\underset{\left\{
						g,A,B
						\right\}}{
						\text{argmin}
					}\ J(g,A,B),
					&
					\label{eq. argmin of learning}
					\\
					&
					\text{where }
					J(g,A,B)
					:= 
					\lambda_1
					\| 
					AG_x
					+
					BU
					-
					G_y
					\|_F^2
					&\nonumber
					\\
					&
					\hspace{16mm}+
					\lambda_2
					\| 
					W
					(
					AG_x
					+
					BU
					)
					- Y \|_F^2,
					&
					\label{eq. loss of initial training}
					\\
					&
					G_x:=
					[g(x_1)\cdots g(x_{M})],
					\
					U:=
					[u_1\cdots u_{M}],
					&\nonumber
					\\
					&
					G_y:=
					[g(y_1)\cdots g(y_{M})],
					\
					Y
					:=
					[y_1\cdots y_{M}],
					&\nonumber
					\\
					&
					y_k=F(x_k,u_k),
					\
					k=1,\cdots, M,
					&\nonumber
					\\
					&
					W:=
					\left[
					\begin{array}{cc}
						I
						&
						0
					\end{array}
					\right],
					&
					\label{eq. decoder}
					\\[-1ex]
					&
					g:\mathcal{X}\rightarrow \mathbb{R}^{n+N}:
					x_k\mapsto 
					\left[
					\begin{array}{c}
						x_k
						\\
						\tilde{g}(x_k)
					\end{array}
					\right],&
					\label{eq. g subsystemized}
					\\
					&
					\tilde{g}(x_k)= \text{NN}(x_k;w)
					\ (\text{a neural network}).
					&
				\end{align}
			\end{prob}
		\end{screen} 
	\end{table}

	The loss function $J(g,A,B)$ in \eqref{eq. loss of initial training} consists of two terms. The first one multiplied by a hyperparameter $\lambda_1$ accounts for (approximately) minimizing $\|r\|_{L_2}$ in \eqref{eq. norm of r}. 
	The other one with a hyperparameter $\lambda_2$ intends to directly minimize the state-reconstruction error by applying the decoder $W$ to the model prediction in order to compare it with the state $y_k$.
	While the decoder may be also characterized by a neural network in general, the specific structure \eqref{eq. def of obs} of observables, which explicitly includes the state $x_k$, allows an analytical expression of the decoder $W$ in \eqref{eq. decoder}.

	The nonlinear feature maps $\tilde{g}$ are defined as a fully-connected feed-forward neural network:
	\vspace{-0mm}
	\begin{align}
		&\text{NN}\hspace{-0.3mm}(\hspace{-0.5mm}x_k\hspace{-0.2mm};\hspace{-0.5mm}w\hspace{-0.5mm})
		\hspace{-1mm}
		:=
		\hspace{-1mm}
		\sigma 
		\hspace{-0.7mm}
		\left(
		\Theta_{l} 
		\sigma
		\hspace{-0.5mm}
		\left(
		\hspace{-0.5mm}\cdots\hspace{-0.5mm} 
		\Theta_2
		\sigma
		\hspace{-0.5mm}
		\left(
		\Theta_1 x_k \hspace{-1mm}+\hspace{-0.5mm} b_1
		\right)
		\hspace{-0.7mm}+\hspace{-0.8mm} b_2
		\hspace{-0.5mm}\cdots\hspace{-0.5mm}
		\right)
		\hspace{-0.7mm}+\hspace{-0.7mm} b_l
		\right)\hspace{-1mm},
		&
		\label{eq. NN}
		\\
		&
		w:=\{ \Theta_{i}, b_i \}_{i=1}^l,
		&
	\end{align}
	where $\Theta_i$, $b_i$, and $\sigma$ are a kernel, a bias, and an activation function, respectively.
	In this paper, we only consider continuous activation functions so that \eqref{eq. NN} is continuous.
	\vspace{-0mm}

	\subsection{Second Training: Modification of the Initial Model}
	\vspace{-0mm}
	\label{section. second training}
	\subsubsection{State-Prediction Accuracy}
	\ \\
	Characterizing observables by a neural network allows greater expressivity, which can lead to higher accuracy of the data-driven model if the optimization problem is feasible. 
	On the other hand, from the controller design perspective, including high-order nonlinearities in the observables is not preferable since it may introduce unexpected or undesirable effect on the closed-loop system due to the modeling error, which could even alter the actual closed-loop system unstable, while it is easier for the state prediction to eliminate the modeling error.
	Specifically, the state prediction is implemented as follows:
	\vspace{-0mm}
	\begin{align}
		&x^{\text{est}}_{k+1}
		\hspace{-1mm}=\hspace{-1mm}
		\underset{=W}{
			\underbrace{
				\left[
				\begin{array}{cc}
					\hspace{-1mm}I\hspace{-0.2mm} & \hspace{-0.2mm}0\hspace{-1mm}
				\end{array}
				\right]}}
		\hspace{-1mm}
		\left\{
		\hspace{-0.5mm}
		A
		\left[
		\begin{array}{c}
			\hspace{-1mm}x^{\text{est}}_{k}\hspace{-1mm}
			\\
			\hspace{-1mm}\tilde{g}(x^{\text{est}}_{k})\hspace{-1mm}
		\end{array}
		\right]
		\hspace{-1mm}+\hspace{-1mm}
		Bu_k
		\hspace{-0.5mm}
		\right\}\hspace{-1mm},
		k=0,\hspace{-0.5mm}1,\cdots\hspace{-0.5mm},
		&
		\label{eq. state prediction}
	\end{align}
	where $x^\text{est}_k$ denotes the state prediction at time $k$ and $x^\text{est}_0=x_0$ is given.
	From \eqref{eq. state prediction}, the modeling error in the state prediction is evaluated as:
	\vspace{-0mm}
	\begin{align}
		x_{k+1}
		\hspace{-1mm}=\hspace{-1mm}
		\left[
		\begin{array}{cc}
			\hspace{-1mm}I\hspace{-0.5mm} & \hspace{-0.5mm}0\hspace{-1mm}
		\end{array}
		\right]
		\hspace{-1mm}
		\left\{
		\hspace{-1mm}
		A
		\left[
		\begin{array}{c}
			\hspace{-1mm}x_{k}\hspace{-1mm}
			\\
			\hspace{-1mm}\tilde{g}(x_{k})\hspace{-1mm}
		\end{array}
		\right]
		\hspace{-1mm}+\hspace{-1mm}
		Bu_k
		\hspace{-1mm}
		\right\}
		\hspace{-1mm}+\hspace{-1mm}
		\left[
		\begin{array}{cc}
			\hspace{-1mm}I\hspace{-0.5mm} & \hspace{-0.5mm}0\hspace{-1mm}
		\end{array}
		\right]
		\hspace{-1mm}
		r(x_k\hspace{-.2mm},\hspace{-0.2mm}u_k).
		\label{eq. state prediction exact relation}
	\end{align}
	
	\vspace{3mm}
	\begin{prop}
		\label{prop. state prediction}
		Let $x_k$ and $u_k$ be arbitrary.
		There exist $\tilde{g}$, $A_1\in \mathbb{R}^{n\times n}$, $A_2\in \mathbb{R}^{n\times N}$ and $B_1\in \mathbb{R}^{n\times p}$ s.t.
		\vspace{-0mm}
		\begin{align}
			F(x_k,u_k)=
			A_1 x_k + A_2 \tilde{g}(x_k) + B_1 u_k,
			\label{eq. condition for perfect state prediction}
		\end{align}
		if and only if
		\vspace{-0mm}
		\begin{align}
			\left[
			\begin{array}{cc}
				\hspace{-1mm}I\hspace{-0.5mm} & \hspace{-0.5mm}0\hspace{-1mm}
			\end{array}
			\right]
			r(x_k,u_k)=0,
		\end{align}
		i.e., \eqref{eq. state prediction} has no state prediction error.
	\end{prop}
	\begin{proof}
		Let $[A_1\ A_2]\in \mathbb{R}^{n\times (n+N)}$ and $B_1\in \mathbb{R}^{n\times p}$ be the first $n$ rows of $A$ and $B$ in \eqref{eq. state prediction exact relation}, respectively.
		From \eqref{eq. governing eq}, the equation \eqref{eq. state prediction exact relation} reads
		\vspace{-0mm}
		\begin{align}
			F(x_k,u_k)
			\hspace{-1mm}=\hspace{-1mm}
			A_1 x_k \hspace{-1mm}+\hspace{-1mm} A_2 \tilde{g}(x_k) \hspace{-1mm}+\hspace{-1mm} B_1 u_k
			\hspace{-1mm}+\hspace{-1mm}
			\left[
			\begin{array}{cc}
				\hspace{-1mm}I\hspace{-0.7mm} & \hspace{-0.7mm}0\hspace{-1mm}
			\end{array}
			\right]\hspace{-1mm}
			r(x_k,u_k),
		\end{align}
		which implies the statement of the proposition.
	\end{proof}
	\vspace{3mm}

	Note that there exist $\tilde{g}$, $A_1$, $A_2$, and $B_1$ that satisfy \eqref{eq. condition for perfect state prediction} with $A_2=0$ or $\tilde{g}(x_k)\equiv 0$ if and only if the original dynamics \eqref{eq. governing eq} is linear.

	\vspace{3mm}
		\begin{col}
			\label{col. state prediction}
			If \eqref{eq. governing eq} is of the linear form w.r.t. input: $F(x_k,u_k)=\tilde{F}(x_k)+\tilde{B}u_k$ ($\tilde{B}\in \mathbb{R}^{n\times p}$),
			including an enough number of different features to reconstruct $\tilde{F}$ is necessary and sufficient for \eqref{eq. state prediction} to be able to achieve zero state-prediction error.
			In case of control-affine dynamics s.t. $\tilde{B}=\tilde{B}(x_k)$ is dependent on $x_k$, the same argument holds if the state-prediction is implemented with $u_k\equiv 0$. 
		\end{col}
	\vspace{3mm}

		Although the specific model structure we adopt in \eqref{eq. Koopman model 1} limits the validity of accurate state-prediction to certain classes of dynamics as seen in Corollary \ref{col. state prediction}, it is considered as a preferable property from the controller design perspective since it allows to utilize linear systems theories in the new coordinates $[x_k^\tr\ \tilde{g}(x_k)^\tr]^\tr$  to control possibly nonlinear dynamics.
		However, it should be emphasized that the choice of model structure is a trade-off relation between the simple and practical controller designs and applicability of the method to actual nonlinear systems, i.e., such linear controller designs in the embedded space will result in poor control performance if the model has a large modeling error.

	\subsubsection{Accuracy in Terms of Closed-Loop Dynamics}
	\ \\
	While the accuracy of the state prediction of the model \eqref{eq. Koopman model 1} may be evaluated by Proposition \ref{prop. state prediction}, its accuracy in terms of controller design is characterized in a different way.
	Note that the system to be controlled by the feedback controller $\bm{K}$ in \eqref{eq. controller} is assumed to be a linear time-invariant system:
	\vspace{-0mm}
	\begin{align}
		\xi_{k+1} = A\xi_k + Bu_k,
		\ \ 
		\xi_k \in \mathbb{R}^{n+N},
		\label{eq. control model}
	\end{align}
	and we can only ensure properties of the closed-loop system:
	\vspace{-0mm}
	\begin{align}
		\xi_{k+1} 
		&= 
		(A+B\bm{K})\xi_{k}
		=
		(A+B\bm{K})^{k+1}\xi_{0}.
		&
	\end{align}

	Clearly, \eqref{eq. control model} is identical to the Koopman control model \eqref{eq. Koopman model 1} if $\|r\|_{L_2}=0$ and $\xi_0=[x_0^\tr\ \tilde{g}(x_0)^\tr]^\tr$.
	However, in general cases where $r(x_k,u_k)\not\equiv 0$, the modeling error may persist at any time and accumulate as follows:
	\vspace{-0mm}
	\begin{align}
		\left[
		\begin{array}{c}
			\hspace{-1mm}x_{k+1}\hspace{-1mm}
			\\
			\hspace{-1mm}\tilde{g}(x_{k+1})\hspace{-1mm}
		\end{array}
		\right]
		&=
		(A+B\bm{K})^{k+1}
		\left[
		\begin{array}{c}
			\hspace{-1mm}x_{0}\hspace{-1mm}
			\\
			\hspace{-1mm}\tilde{g}(x_{0})\hspace{-1mm}
		\end{array}
		\right]
		&\nonumber
		\\
		&\hspace{-15mm}
		+
		\sum_{i=0}^{k} (A+B\bm{K})^i r
		\left(
		x_{k-i}, \bm{K}
		\left[
		\begin{array}{c}
			\hspace{-1mm}x_{k-i}\hspace{-1mm}
			\\
			\hspace{-1mm}\tilde{g}(x_{k-i})\hspace{-1mm}
		\end{array}
		\right]
		\right),
		&
		\label{eq. error accumulation of control model}
	\end{align}
	in which \eqref{eq. controller} is substituted.
	The second term of the r.h.s. of \eqref{eq. error accumulation of control model} represents the discrepancy between \eqref{eq. control model} and \eqref{eq. Koopman model 1}, which directly leads to degradation of controller performance.
	Since the error propagation in \eqref{eq. error accumulation of control model} reflects all components of $r(x_k,u_k)$ unlike the state prediction 
	as in \eqref{eq. state prediction exact relation},
	the actual closed-loop system could greatly suffer from the modeling error depending on the design and/or the dimensions of $\tilde{g}$, as illustrated in the example of Section \ref{section. motivating example}.

        \begin{figure}[t]
            \centering
            \includegraphics[width=\columnwidth]{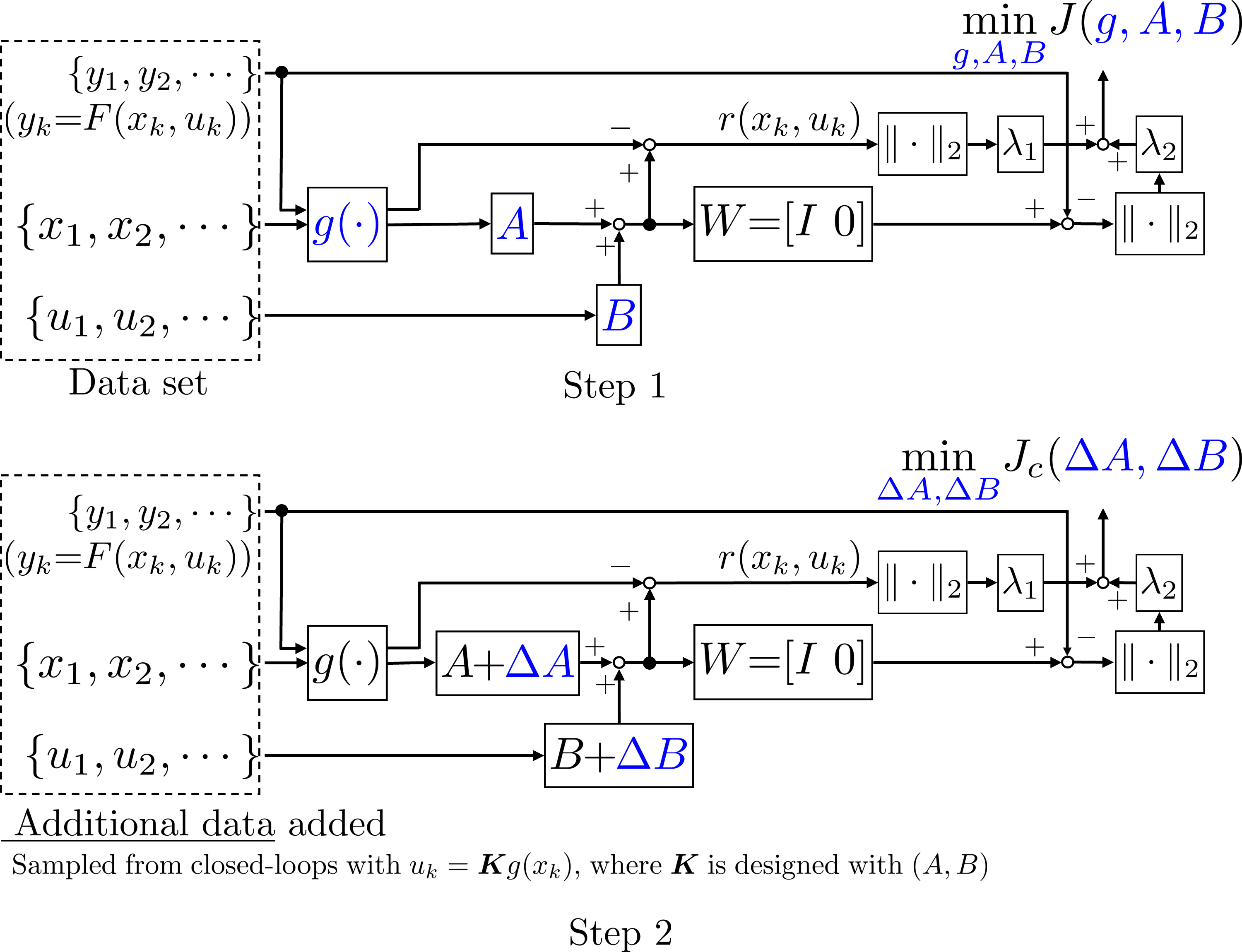}
            \caption{Schematic of the learning procedure of the proposed method. In Step \ref{prob. second training}, constraints $\| \Delta A\|\leq \epsilon_A$ and $\| \Delta B\|\leq \epsilon_B$ are imposed.
            }
            \label{fig. flow_chart}
        \end{figure}

	In the proposed method, considering that the controller performance of the Koopman control model \eqref{eq. Koopman model 1} can deteriorate due to the modeling error,
	the second training process formulated as Step \ref{prob. second training} follows Step \ref{problem. initial training} in case the control objective is not achieved by the initially learned model.
	Specifically, data points sampled from a closed-loop system formed by \eqref{eq. controller}, 
		for which the initially learned $A$ and $B$ are used to design a controller, are 
		added to the initial data set
	so that the second learning problem can directly minimize the modeling error in the regime of closed-loop dynamics \eqref{eq. error accumulation of control model}.

	The modification of the initial model is implemented by updating matrices $A$ and $B$ to
	$A+\Delta A$ and $B+\Delta B$, where we impose the constraint that the modified model retains dynamics close to the initial one. 
	This constraint prevents the additional learning, which retrains the model with 
	additional
	closed-loop data, from degrading the high predictive accuracy obtained in the initial training.
	Specifically, 
	$\Delta A$ and $\Delta B$ have the constraints that their induced 2-norms $\| \Delta A \|:=\text{sup}_{\|x\|=1}\| \Delta A x \|_2$ and $\|\Delta B\|$ are bounded by hyperparameters $\epsilon_A$ and $\epsilon_B$, respectively.
		Finally, the controller gain is recomputed with the updated parameters $(A+\Delta A,B+\Delta B)$.
        The schematic of the learning procedure of the proposed method is shown in Fig. \ref{fig. flow_chart}.
	\begin{table}
		\centering 
		\begin{screen}
			\begin{prob}
				\label{prob. second training}
				\rm{} (Modification of the initial model)
				\\
				Given $g$, $A$, $B$, and a controller gain $\bm{K}$ that is designed with ($A$, $B$), find $\Delta A$ and $\Delta B$ s.t.
				\vspace{-0mm}
				\begin{align}
					&\left\{
					\Delta A,\ \Delta B
					\right\}:=
					\underset{\left\{
						\Delta A,\Delta B
						\right\}}{
						\text{argmin}
					}\ J_c(\Delta A,\Delta B),
					&
					\label{eq. argmin of learning 2}
					\\
					&
					\text{subject to:}	
					\hspace{5mm}
					\| \Delta A \|\leq \epsilon_A,
					&
					\\
					&\hspace{20mm}
					\| \Delta B \|\leq \epsilon_B,
					&
					\\[-1ex]
					&\text{where}&\nonumber
					\\[-1ex]
					&J_c(\Delta A,\Delta B)
					\hspace{-1mm}:=\hspace{-1mm} 
					\hspace{0mm}
					\lambda_1
					\| 
					(A \hspace{-1mm}+\hspace{-1mm} \Delta A)
					G_x
					\hspace{-0.5mm}+\hspace{-0.5mm}
					(B \hspace{-1mm}+\hspace{-1mm} \Delta B)U
					\hspace{-0.5mm}-\hspace{-0.5mm}
					G_y
					\|_F^2
					&\nonumber
					\\
					&\hspace{-1mm}
					+\hspace{-1mm}
					\lambda_2
					\| 
					W
					(
					(A \hspace{-1mm}+\hspace{-1mm} \Delta A)
					G_x
					\hspace{-0.5mm} + \hspace{-0.5mm} 
					(B \hspace{-1mm}+\hspace{-1mm} \Delta B)U
					)
					- Y \|_F^2,
					&
					\label{eq. loss of second training}
					\\
					&
					y_k=F(x_k,u_k),
					&\nonumber
				\end{align}
					and part of input data is generated by the feedback law $u_k=\bm{K}g(x_k)$. 
			\end{prob}
		\end{screen} 
		\vspace{-0mm}
	\end{table}

	\subsection{Characteristics to Measure the Accuracy}
		As stated in Proposition \ref{prop. invariant subspace and existence of finite dimensional K}, the existence of an invariant subspace under the action of the Koopman operator is an essential factor for the modeling procedure, which determines if there exists a finite-dimensional model that can embed the original dynamics linearly in the subspace spanned by observables.
		In this regard, in addition to the modeling error defined in \eqref{eq. def of modeling error}, one can also consider evaluating to what extent the subspace in which the model coordinates are defined is close to being invariant to measure the accuracy of modeling.
		Specifically, given some finite-dimensional subspace $\mathcal{F}_\text{sub}\subset\mathcal{F}$, it is obvious by Proposition \ref{prop. invariant subspace and existence of finite dimensional K} that if $\mathcal{F}_\text{sub}$ is shown to be invariant under the action of $\mathcal{K}$, the realization of an exact model without any modeling error is possible with an arbitrary choice of observables from $\mathcal{F}_\text{sub}$. Also, finding such a model should be relatively easy since it is reduced to a linear regression problem after we choose observables from $\mathcal{F}_\text{sub}$.

		Therefore, the invariance property of the chosen subspace may be considered as a more general characteristic to evaluate the modeling procedure itself than the modeling error $r(x,u)$ defined in \eqref{eq. def of modeling error}.		 
		In fact, if we change the observables of the Model 2 in the example of Section \ref{section. motivating example} from $[x_{k}\ {x^2_k}e^{-x_{k}}\ x^2_k]^\tr$ to $[\alpha x_{k}\ \alpha {x^2_k}e^{-x_{k}}\ \alpha x^2_k]^\tr$
		for some $\alpha \in \mathbb{R}$, norms of the modeling error $r(x,u)$ result in different values as shown in Fig. \ref{fig. ex1 different feature maps} despite the fact that both models should be considered as essentially the same ones since the choices of observables are the same up to multiplicity by $\alpha$.
		This result may be also inferred from the following error bounds analysis.
		\vspace{3mm}
		\begin{prop}
			\label{prop. an error bound}
			Let $g$ be observables defined in \eqref{eq. g subsystemized} and suppose $\mu_x$ and $\mu_u$ are positive measures with compact supports s.t. 
			$(\mathcal{X}, \mathcal{A}_x, \mu_x)$ and $(\mathcal{U}, \mathcal{A}_u, \mu_u)$ are measure spaces with any appropriate $\sigma$-algebras $\mathcal{A}_x$ and $\mathcal{A}_u$, respectively.
			Also, let $\mu:=\mu_x \mu_u$ be their product measure.
			If $g$ is measurable and continuous, the following holds:
			\begin{align}
				&\| r \|_{L_2}^2
				\leq 
				\| g\circ F \|_{L_2}^2
				+
				\| [A\ B] \|
				\left\{
				\| [A\ B] \|
				\| h \|_{L_2}^2
				\right.&\nonumber
				\\
				&\left.
				\hspace{2cm}
				+
				2
				\langle 
				\| (g\circ F)(\cdot) \|_2, \| h(\cdot) \|_2
				\rangle_{L_2} 
				\right\},&
				\label{eq. error bound}
			\end{align}
			where $h:\mathcal{X}\times \mathcal{U}\rightarrow \mathbb{R}^{N+p}$ is a vector-valued function s.t.
			\begin{align}
				h(x,u):=
				\left[
				\begin{array}{c}
					g(x)
					\\
					u
				\end{array}
				\right],
			\end{align}
			and
			the function $\| (g\circ F)(\cdot) \|_2:\mathbb{R}^{n+p}\rightarrow \mathbb{R}_{\geq 0}:(x,u)\mapsto \| (g\circ F)(x,u) \|_2$ denotes an $l_2$ norm on $\mathbb{R}^{n+p}$, so does $\| h(\cdot) \|_2$.
		\end{prop}
		\begin{proof}
			See Appendix \ref{appendix proof of prop 3 an error bound}.
		\end{proof}
		\vspace{3mm} 
		Note that the measures $\mu_x$ and $\mu_u$ in Proposition \ref{prop. an error bound} correspond to probability distributions from which we sample data points $x_k$ and $u_k$, respectively.
		\vspace{3mm}
		\begin{remark}
			Given a model of the form \eqref{eq. Koopman model 1}, if we construct another model with the same observables but multiplied by some $\alpha \neq 0$ s.t. $[\alpha x_k^\tr\ \alpha \tilde{g}(x_k)^\tr]^\tr$, error bounds given by Proposition \ref{prop. an error bound} can be different for these essentially same models.
			For instance, by replacing $B$ in \eqref{eq. Koopman model 1} with $\alpha B$, both models represent the exact same difference equation up to multiplicity $\alpha$, while the corresponding error bounds may be different, e.g., $\| g\circ F \|_{L_2}$ and $\| [A\ B] \|$ in \eqref{eq. error bound} become $|\alpha|\| g\circ F \|_{L_2}$ and $\| [A\ \alpha B] \|$, etc.
		\end{remark}
		\vspace{3mm} 
		Thus, it should be emphasized that the modeling error $r(x,u)$ we adopt in this paper is specific to individual learning results and error analyses may not be robust to unessential variations of models.
		However, it is useful to incorporate into the modeling procedure since it depends on all model parameters and optimization problems can be posed in a natural way based on it, as seen in the loss function $J(g,A,B)$ in \eqref{eq. loss of initial training}.
		Devising characteristics that are both appropriate to measure the accuracy of models and useful to formulate the learning problem is left for future work.
	
	\begin{figure}
		\centering
		\begin{subfigure}{0.45\linewidth}
			\centering
			\includegraphics[width=0.95\linewidth]{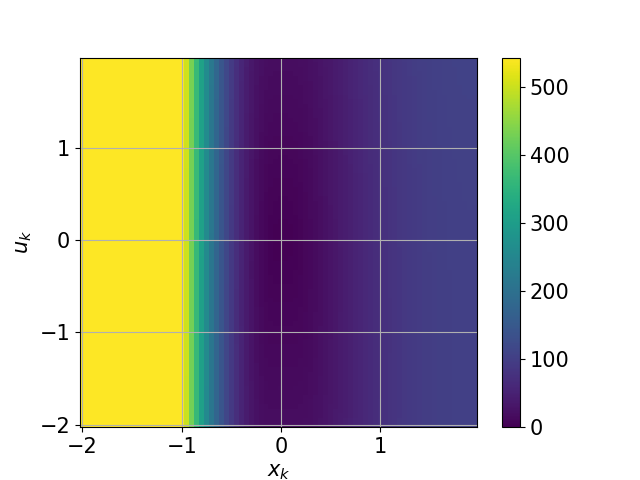}
			\caption{$\alpha=10$.}
		\end{subfigure}
		\begin{subfigure}{0.45\linewidth}
			\centering
			\includegraphics[width=0.95\linewidth]{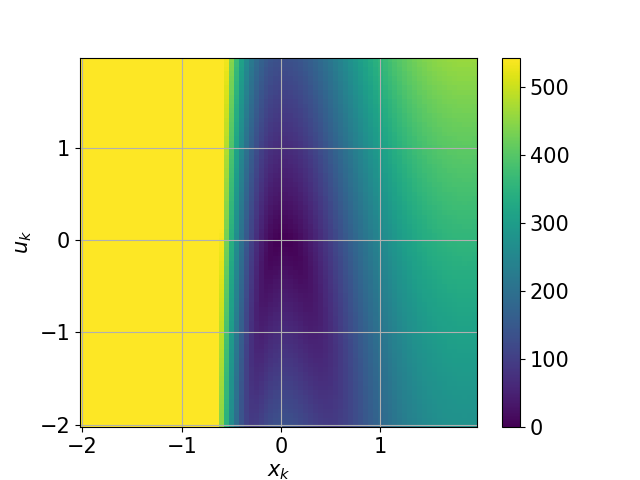}
			\caption{$\alpha=50$.}
		\end{subfigure}
		\caption{Heat maps of $\|r(x_k,u_k)\|_2$ of the Model 2 \eqref{eq. r2} leaned with observables $g(x_k)=[\alpha x_k\ \alpha {x^2_k}e^{-x_{k}}\ \alpha x^2_k]^\tr$. Also see Fig. \ref{subfig. ex1 heat map r2} for comparison.}
		\label{fig. ex1 different feature maps}
	\end{figure}
	
	\vspace{-0mm}
	\subsection{Restricting Inputs of the Data Set}
	\vspace{0mm}
	\label{section. using restricted inputs}
	In addition to the modification of the initially learned model, we also propose a data sampling strategy that generates the input $u_k$ deterministically such that $(u_k,u_{k+1},\cdots)$ will be a sequence of data points sampled from continuous functions. 
	From the modeling perspective, it is in general advisable to sample data points $(x_k, u_k)$ randomly, i.e., from some probability distribution.
	Indeed, assuming that data points $(x_k, u_k)$ are i.i.d. random variables, it is confirmed that minimizing the loss function $J(g,A,B)$ in \eqref{eq. loss of initial training} corresponds to minimizing the norm \eqref{eq. norm of r} of the modeling error $r$ since the first term of $J(g,A,B)$ is related to $\|r\|_{L_2}$ as:
	\vspace{-0mm}
	\begin{align}
		&
        \cfrac{1}{M}
		\| 
		AG_x
		\hspace{-1mm}+\hspace{-1mm}
		BU
		\hspace{-1mm}-\hspace{-1mm}
		G_y
		\|_F^2
		=
		\cfrac{1}{M} \sum_{k=1}^{M} \| r(x_k,u_k) \|_2^2
		&\nonumber
\\
		\overset{\text{a.s.}}{
			\rightarrow
		}
		&
        \int_{\mathcal{X}\times \mathcal{U}}
		\hspace{-2mm}
		\| r(x,u) \|_2^2 dxdu
		\ \ (M\hspace{-1mm}\rightarrow\hspace{-1mm} \infty )
   =
		\|r\|_{L_2}^2,
		&
		\label{eq. relation of loss and the norm of r}
	\end{align}
	where the almost sure convergence follows from the strong law of large numbers.

	However, it is often the case that sampling a very large number of data points across the entire space is not practical due to limitations of experimental or computational resources.
	For non-autonomous systems, the space from which the data $(x_k,u_k)$ is sampled is the product space $\mathcal{X}\times \mathcal{U}$, not the original state space $\mathcal{X}$, and it is especially difficult to sample enough data in applications.
	As a result, learned models may be overfitted or biased, which do not necessarily lead to the modeling error profile consistent with
	the analysis in 
	\eqref{eq. relation of loss and the norm of r}.

	In this paper, we propose to only use deterministic $u_k$ sampled from continuous functions.
	Since the solution $x(t)$ to \eqref{eq. governing eq in ODE} is assumed to be continuous and $\tilde{g}(x)$ defined by \eqref{eq. NN} is also continuous, the control inputs \eqref{eq. controller} are always discretized points sampled from continuous functions. 
	Hence, for the same reason as that of the modification of the initial model in Section \ref{section. second training}, 
	we only use $u_k$ sampled from continuous functions so that Steps \ref{problem. initial training} and \ref{prob. second training} can minimize losses over possible regime of dynamics realized by the controller \eqref{eq. controller} only.

	\vspace{-0mm}
	\section{Numerical Examples}
	\vspace{-0mm}
	\label{section. numerical examples}
	In this section, 
		state-prediction is only implemented with $u_k\equiv 0$ for simplicity.
	Also,
	the control objective is defined as stabilizing the system at the origin while minimizing the cost and LQR is used for the controller design.
	
	\vspace{-0mm}
	\subsection{Simple Pendulum}
	\vspace{-0mm}
	
	\begin{figure}[b]
		\centering
		\begin{subfigure}{0.45\linewidth}
			\centering
			\includegraphics[width=0.95\linewidth]{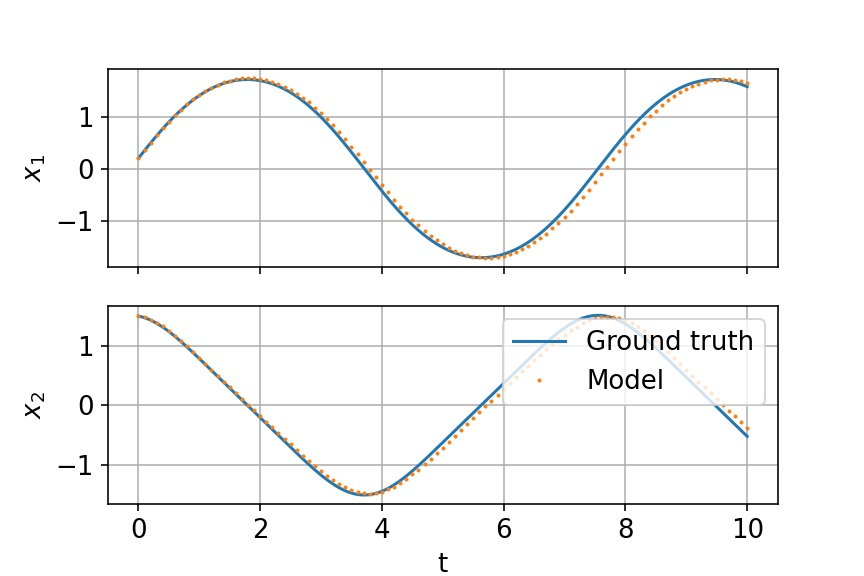}
			\caption{State prediction of the model trained with \eqref{eq. uk random}.}
			\label{subfig. pendulum pred random}
		\end{subfigure}
		\begin{subfigure}{0.45\linewidth}
			\centering
			\includegraphics[width=0.95\linewidth]{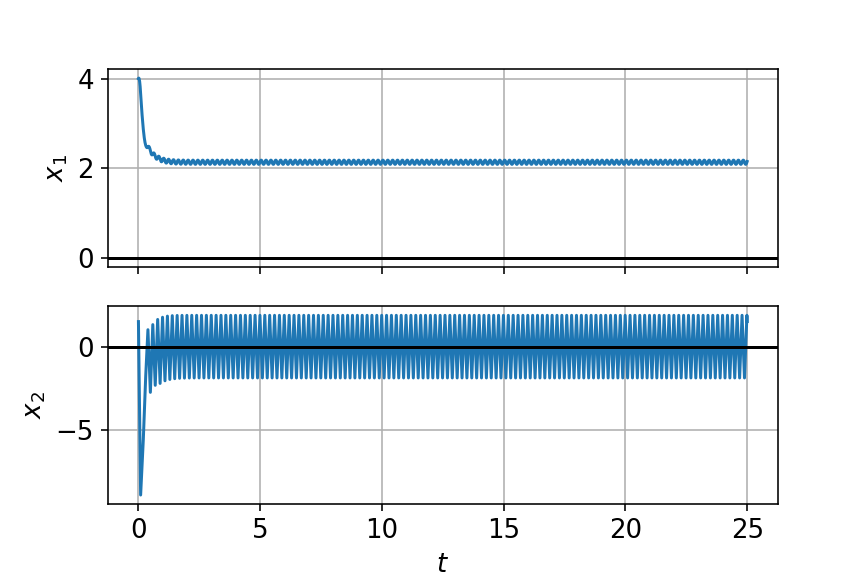}
			\caption{Controller performance of the model trained with \eqref{eq. uk random}.}
			\label{subfig. pendulum cl random}
		\end{subfigure}
		\begin{subfigure}{0.45\linewidth}
			\centering
			\includegraphics[width=0.95\linewidth]{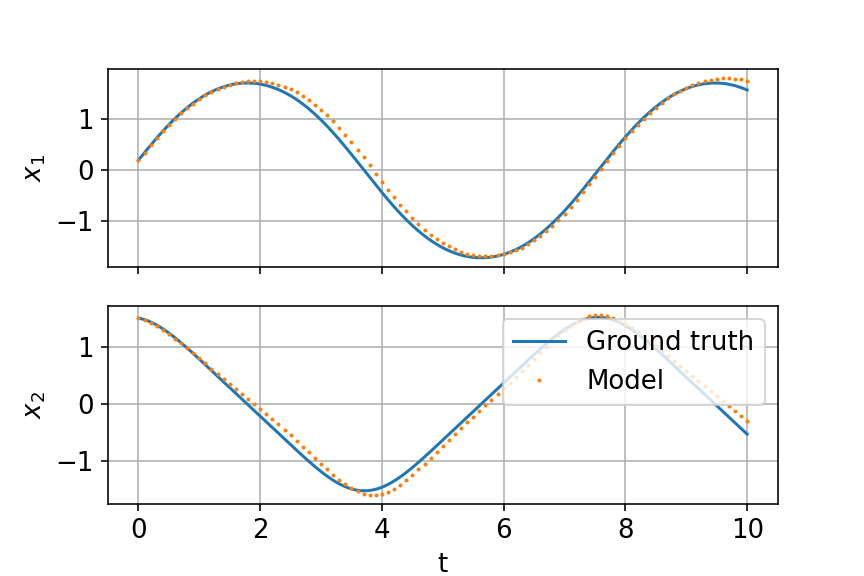}
			\caption{State prediction of the model trained with \eqref{eq. uk cos}.}
			\label{subfig. pendulum pred cos}
		\end{subfigure}
		\begin{subfigure}{0.45\linewidth}
			\centering
			\includegraphics[width=0.95\linewidth]{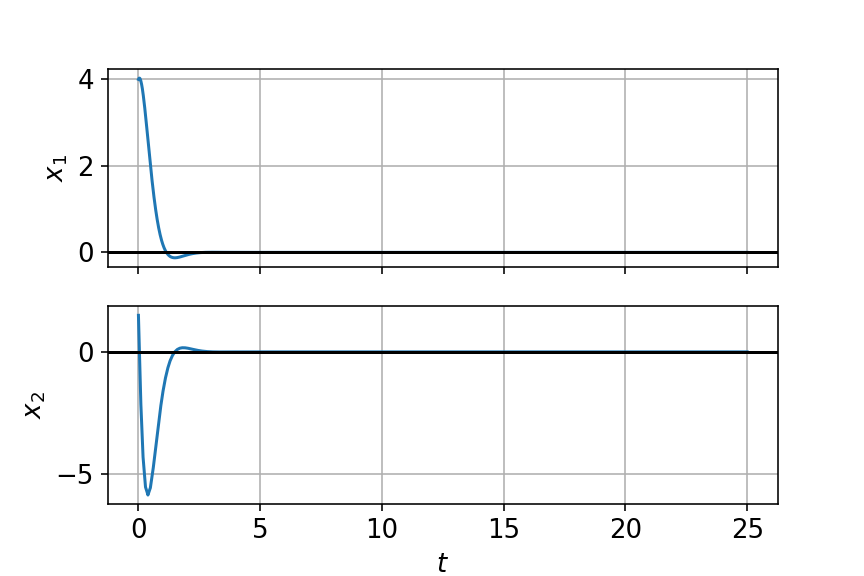}
			\caption{Controller performance of the model trained with \eqref{eq. uk cos}.}
			\label{subfig. pendulum cl cos}
		\end{subfigure}
		\caption{\em Results of the simple pendulum.}
		\label{fig. pendulum}
		\vspace{-0mm}
	\end{figure}
	
	The simple pendulum is considered as the first example:
	\vspace{-0mm}
	\begin{align}
		\ddot{\theta} = -\sin\theta + u,\ \ (x_1:=\theta, x_2:=\dot{\theta}).
		\label{eq. pendulum}
	\end{align}

	We collect 300 data sets generated by \eqref{eq. pendulum}, each of which consists of a single trajectory of a length of 50 steps with the sampling period $\Delta t=0.1$ starting from an initial condition $x_0\sim \text{Uniform}[-3,3]^2$.
	The Runge-Kutta method is used to solve \eqref{eq. pendulum} with a step size of 0.01.
	We include one nonlinear feature $\tilde{g}(x_k)\in \mathbb{R}$ ($N=1$) in the model, which is a neural network with a single hidden layer consisting of 10 neurons and the swish function is used as the activation.
	Step \ref{problem. initial training} is implemented in TensorFlow.

	The following two types of $u_k$ are considered to evaluate the efficacy of the sampling strategy in Section \ref{section. using restricted inputs}:
	\vspace{-0mm}
	\begin{align}
		u_k&\sim \text{Uniform}[-1,1],&
		\label{eq. uk random}
		\\[-0.5ex]
		u_k&=
		\cos (\omega_i k\Delta t),
		\ 
		\omega_i:= 20i,
		\
		i=0,1,\cdots,5.
		&
		\label{eq. uk cos}
	\end{align}

	In the proposed method that adopts \eqref{eq. uk cos}, each single trajectory data set is split evenly into six groups $\mathcal{D}_i$ and $u_k$ with $\omega_i$ is included in $\mathcal{D}_i$ ($i=0,1,\cdots,5$).

	Figure \ref{fig. pendulum} shows the results of the models obtained by Step \ref{problem. initial training}, where the state predictions (Figs. \ref{subfig. pendulum pred random} and \ref{subfig. pendulum pred cos}) are implemented according to \eqref{eq. state prediction} and the control simulations (Figs. \ref{subfig. pendulum cl random} and \ref{subfig. pendulum cl cos}) use LQR gains computed with the cost $\sum_{k=0}^{\infty} 100x_{k,1}^2+x_{k,2}^2+u_k^2$.
		Note that if the discretized dynamics $F$ that is considered as equivalent to the original dynamics \eqref{eq. pendulum} can be obtained by the forward Euler discretization, $F$ is linear w.r.t. input and high predictive accuracy is expected by Corollary \ref{col. state prediction} with the use of sufficiently rich feature maps.
	Whereas the state-prediction accuracy barely changes with different types of $u_k$, the controller performance is greatly deteriorated if the model uses the randomly generated input \eqref{eq. uk random}. The controller designed for the model trained with the proposed sampling strategy successfully makes the state converge to the origin.
	Since the control objective is already achieved by the initial model, Step \ref{prob. second training} is not applied in this example.

	\begin{figure}
		\centering
		\begin{subfigure}{0.45\linewidth}
			\centering
			\includegraphics[width=0.95\linewidth]{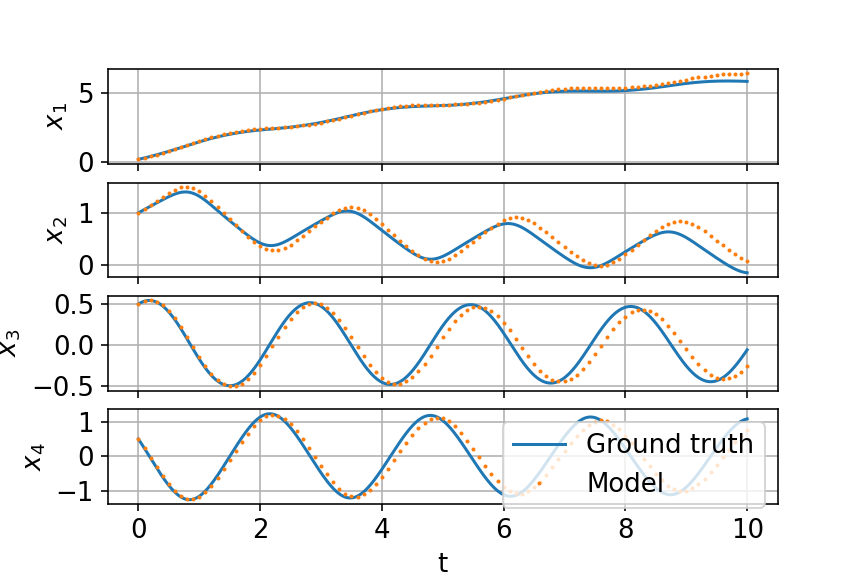}
			\caption{State prediction of the initial model.}
			\label{subfig. invp pred initial}
		\end{subfigure}
		\begin{subfigure}{0.45\linewidth}
			\centering
			\includegraphics[width=0.95\linewidth]{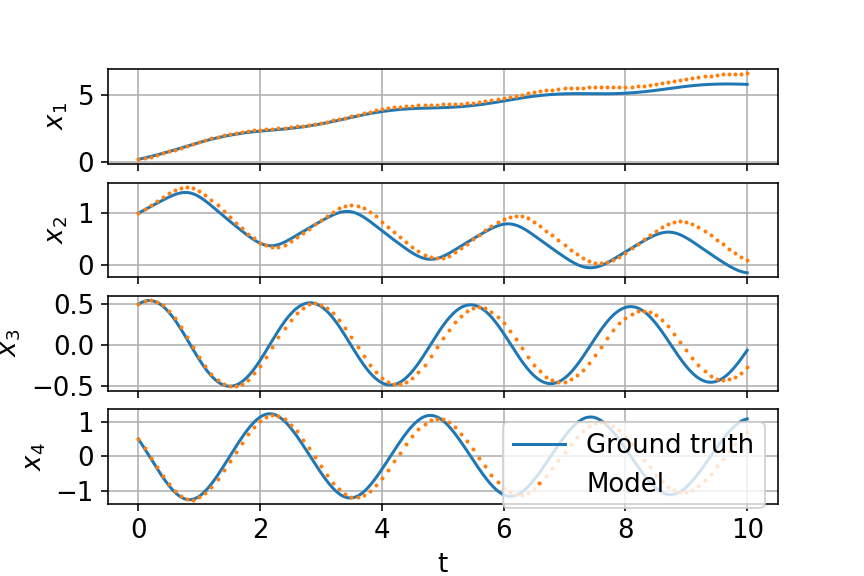}
			\caption{State prediction of the modified model.}
			\label{subfig. invp pred modified}
		\end{subfigure}
		\begin{subfigure}{0.45\linewidth}
			\centering
			\includegraphics[width=0.95\linewidth]{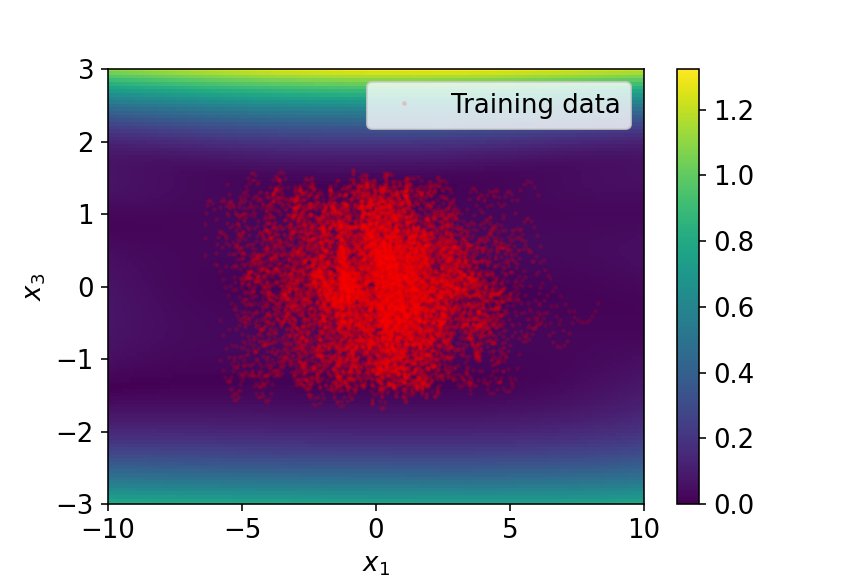}
			\caption{State prediction error of the initial model.}
			\label{subfig. invp error contour initial}
		\end{subfigure}
		\begin{subfigure}{0.45\linewidth}
			\centering
			\includegraphics[width=0.95\linewidth]{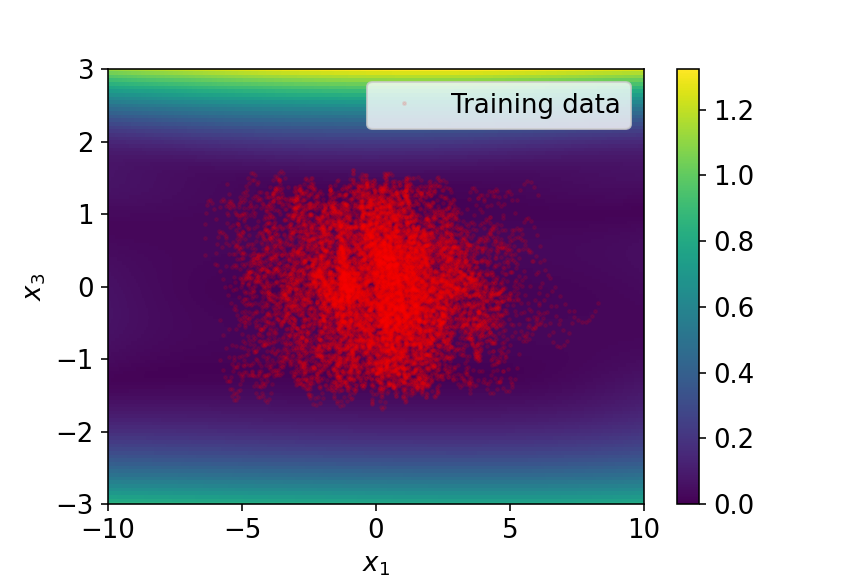}
			\caption{State prediction error of the modified model.}
			\label{subfig. invp error contour modified}
		\end{subfigure}
		\begin{subfigure}{0.45\linewidth}
			\centering
			\includegraphics[width=0.95\linewidth]{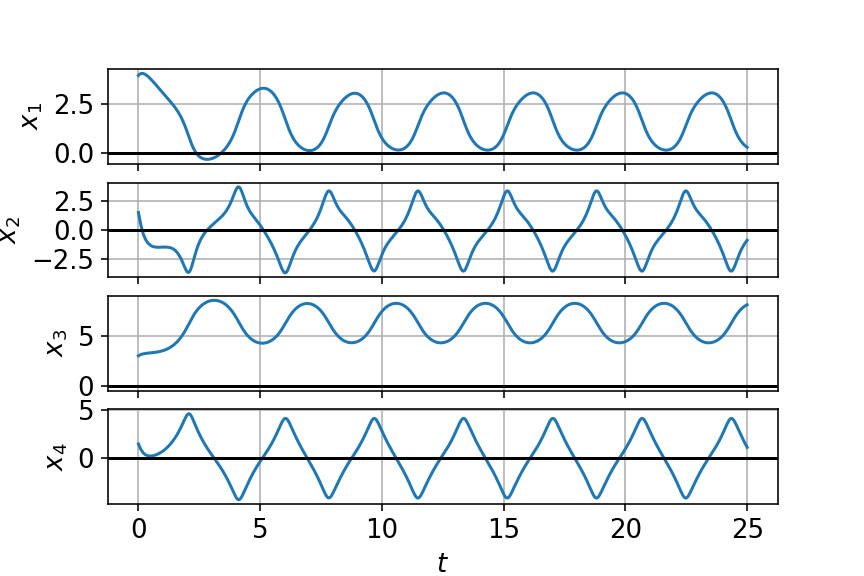}
			\caption{Controller performance of the initial model.}
			\label{subfig. invp control initial}
		\end{subfigure}
		\begin{subfigure}{0.45\linewidth}
			\centering
			\includegraphics[width=0.95\linewidth]{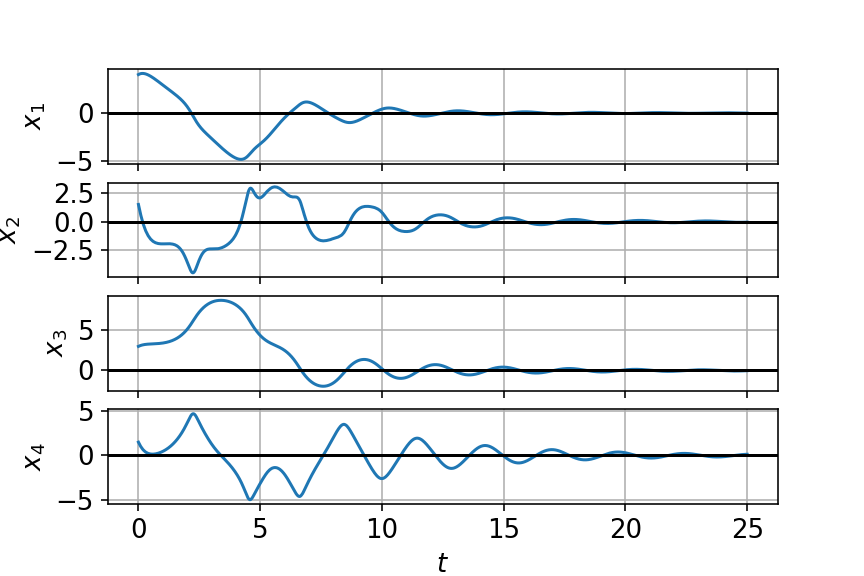}
			\caption{Controller performance of the modified model.}
			\label{subfig. invp control modified}
		\end{subfigure}
		\begin{subfigure}{0.45\linewidth}
			\centering
			\includegraphics[width=0.95\linewidth]{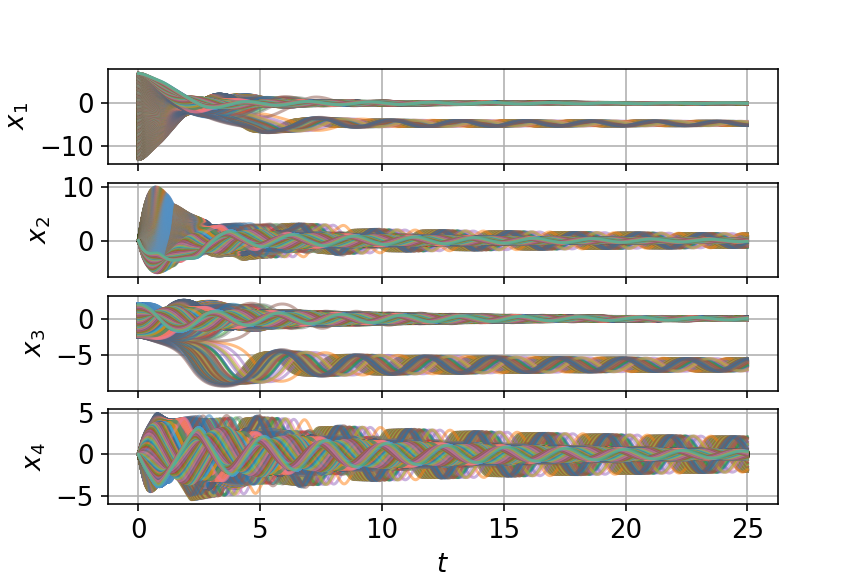}
			\caption{Estimate of basin of attraction of the initial model.}
			\label{subfig. invp basin of attraction initial}
		\end{subfigure}
		\begin{subfigure}{0.45\linewidth}
			\centering
			\includegraphics[width=0.95\linewidth]{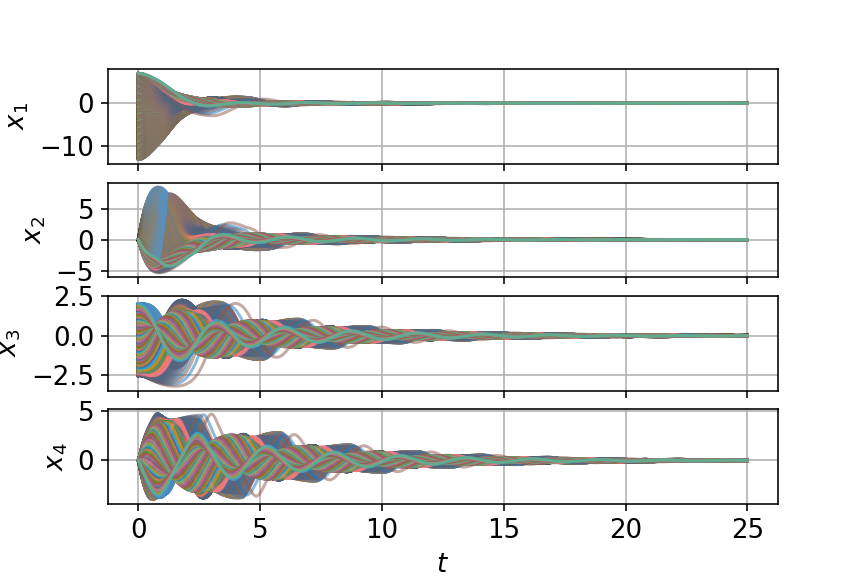}
			\caption{Estimate of basin of attraction of the modified model.}
			\label{subfig. invp basin of attraction modified}
		\end{subfigure}
		\caption{\em Results of the inverted pendulum on a cart. (c),(d): $x_2$ and $x_4$ are fixed to 0. (g),(h): Tested ranges are $x_1(0)\in [-13,6.9]$, $x_2(0)\in [-2.5,2]$ with $x_2(0)=x_4(0)=0$.}
		\label{fig. invp}
		\vspace{-0mm}
	\end{figure}
	
	\vspace{-0cm}
	\subsection{Inverted Pendulum on A Cart}
	\vspace{-0mm}
	The inverted pendulum on a cart is considered as the second example, whose dynamics is given as follows\cite{Data_driven_book}:
	\vspace{-0mm}
	\begin{align}
		\left\{
		\hspace{-2mm}
		\begin{array}{l}
			\dot{x}_1=x_2,
			\\
			\dot{x}_2=
			\cfrac{
				-\hspace{-1mm}m^2L^2g\cos x_3\sin x_3 
				\hspace{-1mm}+\hspace{-1mm} 
				mL^2A(x_2,x_3,x_4) 
				\hspace{-1mm}+\hspace{-1mm}
				mL^2u
			}
			{mL^2(M+m(1-\cos^2 x_3))},
			\\
			\dot{x}_3=x_4,
			\\
			\dot{x}_4=\hspace{-1mm}
			\cfrac{
				\begin{array}{l}
					(m+M)mgL\sin x_3 -mL\cos x_3 A(x_2,x_3,x_4)
					\\
					\hspace{47mm}+mL\cos x_3 u
				\end{array}
			}
			{mL^2(M+m(1-\cos^2 x_3))},
		\end{array}
		\right.
		\label{eq. inverted pendulum on a cart}
	\end{align}
	where 
	$A(x_2,x_3,x_4)=mL{x_4}^2\sin x_3 -\delta x_2$, $m=1$, $M=5$, $L=2$, $g=-10$, and $\delta=1$.
	Step \ref{problem. initial training} is applied with the same conditions as the first example except for the number of hidden layers, which is changed to 25.
	Also, only the deterministic sampling \eqref{eq. uk cos} for $u_k$ is considered in this example.
	For the controller design, the cost is defined as $\sum_{k=0}^{\infty} 100x_{k,1}^2+x_{k,2}^2+100x_{k,3}^2+x_{k,4}^2+u_k^2$.
		The governing equation \eqref{eq. inverted pendulum on a cart} can yield its equivalent discretized dynamics as the control-affine form with the forward Euler and Corollary \ref{col. state prediction} is applicable to the state-prediction with $u_k\equiv 0$.

	While the initially learned model has reasonable predictive accuracy as in Fig. \ref{subfig. invp pred initial}, the controller performance suffers from the undesirable modeling error effect, which is shown in Fig. \ref{subfig. invp control initial}.
	Thus, Step \ref{prob. second training} is implemented to modify the model, for which we additionally collect data points in the same way as Step \ref{problem. initial training}. The TensorFlow Constrained Optimization module\cite{tfco_paper} is used to solve the optimization problem in Step \ref{prob. second training} with $\epsilon_A=\epsilon_B=0.1$.
		The LQR gain is also recomputed with the updated model parameters $(A+\Delta A, B + \Delta B)$.

	It is shown in Fig. \ref{subfig. invp control modified} that the modified model achieves the control objective.
	As a more quantitative analysis, we estimate the basin of attraction of the closed-loop systems by testing various initial conditions, whose results are shown in Figs. \ref{subfig. invp basin of attraction initial} and \ref{subfig. invp basin of attraction modified}.
	Some initial conditions do not converge to the origin for the closed-loop system formed by the initial model, while it is shown that the given set of initial conditions is a basin of attraction for the one formed by the modified model.
	Moreover, the modified model retains good state-prediction accuracy thanks to the constraints on $\Delta A$ and $\Delta B$ and it is comparable to that of the initial model (Figs. \ref{subfig. invp pred initial} and \ref{subfig. invp pred modified}).
	Finally, the profiles of the state prediction errors are evaluated.
	Figures \ref{subfig. invp error contour initial} and \ref{subfig. invp error contour modified} show the heat maps of $\| [I\ 0]r(x,0) \|_2$ with $x_2=x_4=0$, which are overlaid with the data points used in Step \ref{problem. initial training}. 
	It is confirmed that both models have quite similar error profiles and the good state-prediction accuracy of the initial model is successfully preserved after the modification.

	\vspace{-0cm}
	\section{Conclusion}
	\vspace{-0.2cm}
	A  control-aware learning method is proposed along with a simple but effective data sampling strategy for  Koopman operator-based control. The initial learning of the 
	observables and the operator matrices is augmented by a second step to improve the controller performance.
	The use of deterministically-sampled input data from continuous functions improves the controller performance and the proposed two-stage learning contributes to improved closed-loop behavior by updating the operator matrices while retaining high state-prediction accuracy obtained by the initially learned model.

        \section{Acknowledgement}
        This work was funded by APRA-E under the project \textit{SAFARI: Secure Automation for Advanced Reactor Innovation}.
	
	\appendix
	\subsection{Proof of Proposition \ref{prop. invariant subspace and existence of finite dimensional K}}
	\label{appendix proof of prop 1}
		Suppose $\text{span}(g_1\cdots, g_D)$ is an invariant subspace under the action of the Koopman operator $\mathcal{K}$, which implies
		\begin{align}
			\mathcal{K}g 
			=
			\sum_{j=1}^{D} a_j (\mathcal{K}g_j),
		\end{align}
		is in $\text{span}(g_1\cdots, g_D)$ for $\forall g\in \text{span}(g_1\cdots, g_D)$, where $g=\sum_{j=1}^{D}a_jg_j$ for some $a_j\in \mathbb{R}$.
		By taking $g=g_i$, $i=1,\cdots,D$, we have
		\begin{align}
			&
			\exists k_{i1},\cdots,k_{iD}\in \mathbb{R}
			\text{ s.t. }
			\mathcal{K}g_i
			=
			\sum_{j=1}^{D} k_{ij} g_j,
			\text{ for }i=1,\cdots,D,
			&\nonumber
			\\
			&\Leftrightarrow\ \ 
			\exists K\in \mathbb{R}^{D\times D}
			\text{ s.t. }
			\left[
			\begin{array}{c}
				\mathcal{K}g_1
				\\
				\vdots
				\\
				\mathcal{K}g_D
			\end{array}
			\right]
			=
			K
			\left[
			\begin{array}{c}
				g_1
				\\
				\vdots
				\\
				g_D
			\end{array}
			\right].
			&
		\end{align}

		Conversely, if there exists $K\in \mathbb{R}^{D\times D}$ satisfying \eqref{eq. finite dimensional approx of K}, for $\forall g=\sum_{i=1}^{D}a_i g_i\in \text{span}(g_1,\cdots,g_D)$,
		\begin{align}
			\mathcal{K}g
			=&
			\sum_{i=1}^{D} a_i (\mathcal{K}g_i)
			&\nonumber
			\\
			=&
			\sum_{i=1}^{D} a_i 
			\sum_{j=1}^{D} k_{ij} g_j
			\ 
			(\text{for some }k_{ij}\in \mathbb{R})
			&\nonumber
			\\
			=&
			\sum_{j=1}^{D}
			\left(
			\sum_{i=1}^{D} a_i k_{ij}
			\right)
			g_j
			&\nonumber
			\\
			\in&
			\text{span}(g_1,\cdots,g_D).
			&
		\end{align}
		\hspace{0.9\linewidth}$\blacksquare$

	\subsection{Proof of Proposition \ref{prop. an error bound}}
	\label{appendix proof of prop 3 an error bound}
		Since $\mu_x$ and $\mu_u$ have compact supports and $g$ is both measurable and continuous, $g\in L_p(\mathcal{X}, \mathcal{A}_x, \mu_x)$ for all $1\leq p < \infty$.
		Specifically, we have $g\in L_1(\mathcal{X}, \mathcal{A}_x, \mu_x)$ and $g\in L_2(\mathcal{X}, \mathcal{A}_x, \mu_x)$.
		The function $h$ is well-defined as $h\in L_2$. Indeed,
		\begin{align}
			&\int_{\mathcal{X}\times \mathcal{U}} 
			\left\|
			\left[
			\begin{array}{c}
				g(x)
				\\
				u
			\end{array}
			\right]
			\right\|_2^2
			\mu(dxdu) 
			&\nonumber
			\\
			&=\hspace{-1mm}
			\underset{=:a}{\underbrace{
					\int_{\mathcal{X}}
					\hspace{-1mm}
					\left\|
					g(x)	
					\right\|_2^2
					\mu_x(dx)
			}}
			\underset{=:b}{\underbrace{
					\int_{\mathcal{U}}
					\hspace{-1mm}
					\mu_u(du)
			}}
			\hspace{-0.5mm}+\hspace{-1.5mm}
			\underset{=:c}{\underbrace{
					\int_{\mathcal{X}}
					\mu_x(dx)
			}}
			\underset{=:d}{\underbrace{
					\int_{\mathcal{U}}
					\left\|
					u	
					\right\|_2^2
					\mu_u(du)
			}}
			&\nonumber
			\\
			&<
			\infty,
			&
		\end{align}
		where the last inequality holds since $a=\| g \|_{L_2}^2<\infty$ and $b$, $c$, and $d$ are also finite since the supports of $\mu_x$ and $\mu_u$ are bounded.
		Using the triangle inequality, we have
		\begin{align}
			\| r(x,u) \|_2
			&=
			\left\| g(F(x,u)) - [A\ B]
			\left[
			\begin{array}{c}
				g(x)
				\\
				u
			\end{array}
			\right]
			\right\|_2
			&\nonumber
			\\
			&\leq 
			\| g(F(x,u)) \|_2
			+
			\| [A\ B] \|
			\left\| 
			\left[
			\begin{array}{c}
				g(x)
				\\
				u
			\end{array}
			\right]
			\right\|_2,
			&\nonumber
			\\
			\Leftrightarrow
			\ \ 
			\| r(x,u) \|_2^2
			&\leq 
			\| g(F(x,u)) \|_2^2
			&\nonumber
			\\
			&
			\hspace{3mm}
			+
			\| [A\ B] \|
			\left\{
			\| [A\ B] \|
			\left\| 
			\left[
			\begin{array}{c}
				g(x)
				\\
				u
			\end{array}
			\right]
			\right\|_2^2
			\right. 
			&\nonumber
			\\
			&
			\hspace{3mm}
			\left. 
			+2
			\| g(F(x,u)) \|_2
			\left\| 
			\left[
			\begin{array}{c}
				g(x)
				\\
				u
			\end{array}
			\right]
			\right\|_2
			\right\}.
			\label{eq. triangle inequality}
		\end{align}

		Define 
		\begin{align}
			\hat{h}:\mathcal{X}\times \mathcal{U}\rightarrow \mathbb{R}^{n+p}
			:(x,u)\mapsto 
			[A\ B]
			\left[
			\begin{array}{c}
				g(x)
				\\
				u
			\end{array}
			\right].
		\end{align}

		It is confirmed that $\hat{h}\in L_2$ as 
		\begin{align}
			&\int_{\mathcal{X}\times \mathcal{U}} 
			\left\|
			[A\ B]
			\left[
			\begin{array}{c}
				g(x)
				\\
				u
			\end{array}
			\right]
			\right\|_2^2
			\mu(dxdu)
			&\nonumber
			\\
			=&
			\int_{\mathcal{X}\times \mathcal{U}} 
			\sum_{i=1}^N
			\left(
			\sum_{l=1}^{N} a_{i,l} g_l(x)
			+
			\sum_{l=1}^{p} b_{i,l} u_l
			\right)^2
			\mu(dxdu)
			&\nonumber
			\\
			=&
			\sum_{i=1}^{N}
			\left\{
			\sum_{l=1}^{N} 
			a_{i,l}^2 
			\| g_l \|_{L_2}^2
			\int_{\mathcal{U}} \mu_u(du)
			\right. 
			&\nonumber
			\\
			&\hspace{9mm}\left. 
			+
			2\sum_{j=2}^{N} \sum_{l=1}^{j-1}
			a_{i,l} a_{i,j}
			\langle g_l, g_j \rangle_{L_2}
			\int_{\mathcal{U}} 
			\mu_u(du)
			\right. 
			&\nonumber
			\\
			&\hspace{9mm}\left.
			+
			2 \sum_{l=1}^{N} \sum_{j=1}^{p} 
			a_{i,l} b_{i,l} 
			\int_{\mathcal{X}} g_l(x) \mu_x(dx)
			\int_{\mathcal{U}} u_j \mu_u(du)
			\right.
			&\nonumber
			\\
			&
			\hspace{9mm}
			\left. 
			+
			\sum_{l=1}^{p} 
			b_{i,l}^2
			\int_{\mathcal{X}\times \mathcal{U}} 
			u_l^2
			\mu(dxdu)
			\right. 
			&\nonumber
			\\
			&\hspace{9mm}\left.
			+
			2\sum_{j=2}^{p} \sum_{l=1}^{j-1}
			b_{i,l} b_{i,j}
			\int_{\mathcal{X}\times \mathcal{U}} 
			u_l u_j
			\mu(dxdu)
			\right\}
			& 
			\label{eq. proof of prop 3 any integration}
			\\
			<&\infty,
		\end{align}
		where any integration in \eqref{eq. proof of prop 3 any integration} is finite since the supports of $\mu_x$ and $\mu_u$ are bounded and $g\in L_1$, i.e., $g$ is integrable.
		Noticing that $r$ is a linear combination of $g\circ F\in L_2$ and 
		$\hat{h}\in L_2$, we have $r\in L_2$ by the same argument as above.
		Also, $l_2$ norms:
		\begin{align}
			N_1(\cdot):=\| (g\circ F)(\cdot) \|_2
			&:\mathbb{R}^{n+p}\rightarrow \mathbb{R}_{\geq 0}
			&\nonumber
			\\
			&:(x,u)\mapsto \| (g\circ F)(x,u) \|_2,
			&
			\\
			&\hspace{-3.25cm}
			N_2(\cdot):=\| h(\cdot) \|_2:\mathbb{R}^{n+p}\rightarrow \mathbb{R}_{\geq 0}:(x,u)\mapsto \| h(x,u) \|_2,
			&
		\end{align}
		belong to $L_2$ as
		\begin{align}
			\int_{\mathcal{X}\times \mathcal{U}} | N_1(x,u) |^2 \mu(dxdu)
			=&
			\int_{\mathcal{X}\times \mathcal{U}} \| (g\circ F)(x,u) \|_2^2 \mu(dxdu)
			&\nonumber
			\\
			=&
			\| g\circ F \|_{L_2}^2
			&\nonumber
			\\
			<&\infty,&
			\\
			\int_{\mathcal{X}\times \mathcal{U}} | N_2(x,u) |^2 \mu(dxdu)
			=&
			\int_{\mathcal{X}\times \mathcal{U}} \| h(x,u) \|_2^2 \mu(dxdu)
			&\nonumber
			\\
			=&
			\| h \|_{L_2}^2
			&\nonumber
			\\
			<&\infty,&
		\end{align}
		so that the following quantity is well-defined as the inner product of $L_2$:
		\begin{align}
			&\int_{\mathcal{X}\times \mathcal{U}} 
			\| g(F(x,u)) \|_2
			\left\| 
			\left[
			\begin{array}{c}
				g(x)
				\\
				u
			\end{array}
			\right]
			\right\|_2
			dxdu
			&\nonumber
			\\
			=&
			\left\langle 
			\| (g\circ F)(\cdot) \|_2, \| h(\cdot) \|_2
			\right\rangle_{L_2}.
			&
		\end{align}

		Therefore, \eqref{eq. triangle inequality} implies the following: 
		\begin{align} 
			\| r \|_{L_2}^2
			&=
			\int_{\mathcal{X}\times \mathcal{U}} \|r(x,u)\|_2^2 dxdu
			&\nonumber
			\\
			&\leq 
			\int_{\mathcal{X}\times \mathcal{U}} \| g(F(x,u)) \|_2^2 dxdu
			&\nonumber
			\\
			&
			\hspace{5mm}
			+
			\| [A\ B] \|
			\left\{
			\| [A\ B] \|
			\int_{\mathcal{X}\times \mathcal{U}} 
			\left\| 
			\left[
			\begin{array}{c}
				g(x)
				\\
				u
			\end{array}
			\right]
			\right\|_2^2
			dxdu
			\right. 
			&\nonumber
			\\
			&\hspace{5mm}
			\left. 
			+
			2
			\int_{\mathcal{X}\times \mathcal{U}} 
			\| g(F(x,u)) \|_2
			\left\| 
			\left[
			\begin{array}{c}
				g(x)
				\\
				u
			\end{array}
			\right]
			\right\|_2
			dxdu
			\right\}&\nonumber
			\\
			&=
			\| g\circ F \|_{L_2}^2
			+
			\| [A\ B] \|
			\left\{
			\| [A\ B] \|
			\| h \|_{L_2}^2
			\right. 
			&\nonumber
			\\
			&\hspace{2cm}
			\left. 
			+
			2
			\left\langle 
			\| (g\circ F)(\cdot) \|_2, \| h(\cdot) \|_2
			\right\rangle_{L_2} 
			\right\}.
			&\nonumber
		\end{align}
		\hspace{0.9\linewidth}$\blacksquare$
	
	\bibliography{acc2022}  
	\bibliographystyle{ieeetr}

\end{document}